\documentclass[13pt,oneside]{article}
\usepackage{times,amsmath,amsbsy,amssymb,amscd,mathrsfs,graphicx}
\usepackage{amssymb}
\usepackage{amsmath,bm}
\usepackage{amsmath,amsthm}
\usepackage{subcaption}
\usepackage{framed}
\usepackage{xcolor,color}
\usepackage[font=small,labelfont=md,textfont=it]{caption}
\usepackage{floatrow,float}
\usepackage[titletoc, title]{appendix}
\usepackage[colorlinks,linkcolor=blue,citecolor=blue]{hyperref}
\usepackage{longtable}
\usepackage{pgfplots}
\usepackage{diagbox}
\usepackage{booktabs,makecell,multirow}
\usepackage[capitalise, nosort]{cleveref}
\usepackage{algorithm}
\usepackage{algorithmic}
\usepackage{amsmath,bm}
\usepackage{cases}
\usepackage{geometry}
\usepackage{extarrows}
\usepackage{tcolorbox}
\usepackage{bbm}
\usepackage{syntonly}
\usepackage[figuresright]{rotating}
 \usepackage{epstopdf}

\crefname{equation}{}{}
\crefname{lem}{Lemma}{Lemmas}
\crefname{thm}{Theorem}{Theorems}

\newcommand{\R}{\,{\mathbb R}}
\newcommand{\diag}[1]{{\rm diag}\left({#1} \right)}
\newcommand{\triu}[1]{{\rm triu}\left({#1} \right)}

\newcommand{\rank}[0]{ {\mathrm {rank}\,}}

\newcommand{\dual}[1]{\left\langle {#1} \right\rangle}

\newcommand{\myspan}{\,{\rm span}}
\newcommand{\gras}[1]{\,{ \bf Grass}({#1})}
\newcommand{\appker}[1]{\,{\bf ker}_\epsilon({#1})}
\newcommand{\myker}[1]{\, {\bf ker}({#1})}
\newcommand{\dappker}[1]{\,\widehat{{\bf ker}}_\epsilon({#1})}
\newcommand{\row}[1]{\,{\bf row}({#1})}

\newcommand{\argmin}[0]{ {\mathop{{\rm  argmin}}\,}}

\newcommand{\ran}[1]{ {\bf range}({#1})}
\newcommand{\innerA}[1]{\left({#1} \right)_{\!A}}

\newcommand{\inner}[1]{\left({#1} \right)}

\newcommand{\nm}[1]{\left\lVert {#1} \right\rVert}
\newcommand{\snm}[1]{\left\lvert {#1} \right\rvert}

\newcommand{\ssnm}[1]
{
	\left\vert\kern-0.25ex
	\left\vert\kern-0.25ex
	\left\vert
	{#1}
	\right\vert\kern-0.25ex
	\right\vert\kern-0.25ex
	\right\vert
}

\makeatletter
\def\spher@harm#1{%
	\vbox{\hbox{%
			\offinterlineskip
			\valign{&\hb@xt@2\p@{\hss$##$\hss}\vskip.2ex\cr#1\crcr}%
		}\vskip-.36ex}%
}
\def\gshone{\spher@harm{.}}
\def\gshtwo{\spher@harm{.&.}}
\def\gshthree{\spher@harm{.&.&.}}
\let\gsh\spher@harm
\makeatother

\newtheorem{coro}{Corollary}[section]

\newtheorem{Def}{Definition}[section]

\newtheorem{assum}{Assumption}

\newtheorem{lem}{Lemma}[section]

\newtheorem{rem}{Remark}[section]

\newtheorem{thm}{Theorem}[section]

\newtheorem{eg}{Example}

\newcolumntype{I}{!{\vrule width 1,5pt}}
\newlength\savedwidth

\newlength\savewidth

\newcounter{mnote}
\let\oldmarginpar\marginpar
\renewcommand\marginpar[1]
{\-\oldmarginpar[\raggedleft\footnotesize #1]{\raggedright\footnotesize #1}}

\makeatletter\def\@captype{table}\makeatother

\newfloatcommand{capbtabbox}{table}[][\FBwidth]
\usepackage[T1]{fontenc}
\usepackage[utf8]{inputenc}
\usepackage{authblk}
\usepackage[all]{xy}

\usepackage{graphicx}
\usepackage{tikz}
\tikzset{elegant/.style={smooth,thick,samples=50,black}}
\tikzset{eaxis/.style={->,>=stealth}}

\usepackage[misc]{ifsym}

\begin{document}
	\title{
	\Large \bf Robust Kaczmarz methods for nearly singular linear systems\thanks{This work was supported by the National Natural Science Foundation of China (Grant No. 12401402), the Science and Technology Research Program of Chongqing Municipal Education Commission (Grant Nos. KJZD-K202300505), the Natural Science Foundation of Chongqing (Grant No. CSTB2024NSCQ-MSX0329) and the Foundation of Chongqing Normal University (Grant No. 22xwB020).}}

	\author[,1]{Yunying Ke\thanks{Email: 2023110510025@stu.cqnu.edu.cn}}
\author[,1,2]{Hao Luo \thanks{Email: luohao@cqnu.edu.cn; luohao@cqbdri.pku.edu.cn}}
\affil[1]{National Center for Applied Mathematics in Chongqing, Chongqing Normal University, Chongqing, 401331, China}
\affil[2]{Chongqing Research Institute of Big Data, Peking University,  Chongqing, 401121, China}

	\date{\today}
\maketitle

	\begin{abstract}	
	The Kaczmarz method is an efficient iterative algorithm for large-scale linear systems. However, its linear convergence rate suffers from ill-conditioned problems and is highly sensitive to the smallest nonzero singular value. In this work, we aim to extend the classical Kaczmarz to nearly singular linear systems that are row rank-deficient. We introduce a new concept of nearly singular property by treating the row space as an unstable subspace in the Grassman manifold. We then define a related important space called the approximate kernel, based on which a robust kernel-augmented Kaczmarz (KaK) is introduced  via the subspace correction framework and analyzed by the well-known Xu--Zikatanov identity. To get an implementable version, we further introduce the approximate dual kernel and transform KaK into an equivalent kernel-augmented coordinate descent. Furthermore, we develop an accelerated variant and establish the improved rate of convergence matching the optimal complexity of first-order methods. Compared with existing methods, ours achieve uniform convergence rates for nearly singular linear systems, and the robustness has been confirmed by some numerical tests.
\end{abstract}

%	\tableofcontents

%\input{../1-intro}
\section{Introduction}
\label{intro}
%Linear systems are widely used in areas such as image/signal processing \cite{1996Matrix,boyd2018introduction}, economics \cite{judd1998numerical}, and control systems \cite{aastrom2021feedback}, where their solutions help make important decisions. Therefore, developing efficient and accurate solvers has both theoretical and practical value.
%The need for efficient iterative methods has grown with the development of computing, as they often work better than direct methods on large systems \cite{benzi2002preconditioning,davis2011university}. The development of iterative methods has progressed through several key breakthrough—from earlier ones like Jacobi and Gauss–Seidel, to conjugate gradient for symmetric systems \cite{1952Methods}, and Krylov methods like GMRES for nonsymmetric problems \cite{saad_gmres_1986}. Modern versions are now also used in fields including machine learning \cite{saad_flexible_1993}.

The Kaczmarz method \cite{karczmarz1937angenaherte} is a widely used iterative solver for large linear systems. As a row-action approach, it exactly satisfies one equation per iteration while maintaining very low computational cost per step. This method is highly efficient and easy to implement, aligning well with the spirit of the algebraic reconstruction technique \cite{herman2009fundamentals}. Consider a consistent linear system:
\begin{equation}
	\label{Ax=b}
	Ax=b,
\end{equation}
where $b\in \mathbb{R}^{m}$ belongs to $\ran{A}$, the range of $A\in \mathbb{R}^{m \times n}$. Starting from an arbitrary initial vector $x_{0}\in\R^n$, the classical Kaczmarz method \cite{karczmarz1937angenaherte} reads as follows
\begin{equation}
	\label{eq:kacz-intro}
	x_{k+1}= x_k+\frac{b_i-A_{(i)} x_k}{\|A_{(i)}\|^2}A_{(i)}^\top,\quad i = k {\rm\,\,mod\,\,} m,
	%\left\{
	%\begin{aligned}
	%{}&	x_{k}^0=x_k,\\\
	%{}&	x_{k}^{i} = x_k^{i-1}+\frac{b_i-A_{(i)} x_k^{i-1}}{\|A_{(i)}\|^2}A_{(i)}^\top,\quad 1\leq i\leq m,\\
	%{}&	x_{k+1} = x_{k}^m,
	%\end{aligned}
	%\right.
\end{equation} 
where \({A}_{(i)} \) denotes the \( i \)-th row of \({A} \), and \( b_i \) the \( i \)-th component of \( {b} \). Geometrically, at the $k$-th iteration, the  update $x_{k+1}$ is the orthogonal projection of $x_k$ onto the hyperplane $H_i = \{ x \in \mathbb{R}^n:\, A_{(i)}x = b_i \}$, where $i = k {\rm\,\,mod\,\,} m$.

The randomized strategy serves as one effective approach to accelerate the convergence of the Kaczmarz method \cref{eq:kacz-intro}. 
%Empirical evidence indicates that using a randomized row selection sequence for matrix $A$, rather than a deterministic order, typically leads to faster convergence rates \cite{feichtinger_titlenew_1992,2002The}. 
Strohmer and Vershynin \cite{strohmer_randomized_2007} proposed  the randomized Kaczmarz (RK) 
\begin{equation}
	\label{eq:rand-kacz-intro}
	%	\left\{
	%	\begin{aligned}
		%		{}&	x_{k}^0=x_k,\\\
		%		{}&	
		x_{k+1} = x_k+\frac{b_{i_k}-A_{(i_k)} x_k}{\|A_{(i_k)}\|^2}A_{(i_k)}^\top,
		%		\\
		%		{}&	x_{k+1} = x_{k}^m,
		%	\end{aligned}
	%	\right.
\end{equation}
which selects randomly a row index $i_k\in\{1,\cdots,m\}$ and achieves an exponential rate with the contraction factor $1-1/\kappa^2(A)$ (see also \cite{leventhal_randomized_2010}), where $\kappa(A):=\|A\|_F/\sigma^+_{\min}(A)$ and $\sigma^+_{\min}(A)$ are respectively the {\it scaled condition number} \cite{demmel_theprob_1988} and the smallest nonzero singular value of $A$.  After that, RK has found its applications in CT image reconstruction \cite{bicer_trace_2017} and inspired subsequently various extensions, including greedy RK methods \cite{bai2018greedy,gower2021adaptive,su2023convergence,su2024greedy}, randomized sparse Kaczmarz methods \cite{chen2021regularized,schopfer2019linear,zeng_stochastic_2026}, and randomized block Kaczmarz methods \cite{gower2015randomized,necoara2019faster,2014Paved,xiang_randomized_2025,xie2025randomized,zeng2023randomized}.

On the other hand, it is well-known that 
the Kaczmarz method \cref{eq:kacz-intro} is equivalent to applying the coordinate descent (CD) method  to the dual problem 
\begin{equation}\label{eq:dual-prob}
	\min_{y\in \R^{m}}\,g(y)=\frac{1}{2}\|A^{\top}y\|^{2}+b^{\top}y.
\end{equation}
Indeed, for this unconstrained optimization problem, the CD iteration is given by
\begin{equation}
	\label{eq:cd-intro}
	y_{k+1} = y_{k}-sU_{i}U_i^\top \nabla g(y_{k}),\quad i = k {\rm\,\,mod\,\,} m,
	%	\left\{
	%	\begin{aligned}
		%		{}&	y_{k}^0=y_k,\\\
		%		{}&	y_{k}^{i} = y_{k}^{i-1}-sU_{i}U_i^\top \nabla g(y_{k}^{i-1}),\quad 1\leq i\leq m,\\
		%		{}&	y_{k+1} = y_{k}^m,
		%	\end{aligned}
	%	\right.
\end{equation} 
%\begin{equation}\label{eq:cd-intro}
%		y_{k+1}=y_{k}-sU_{i}U_i^\top \nabla g(y_{k}),
%\end{equation}
where $ s>0$ denotes the step size, and $ U_{i} $ the $i$-th column of the identity matrix of order $m$. If $s = \|A_{(i)}\|^{-2}$, then multiplying both sides of \cref{eq:cd-intro} by $ -A^{\top} $ recovers the classical Kaczmarz iteration \cref{eq:kacz-intro}. Nesterov \cite{nesterov_efficiency_2012} proposed a randomized version of \cref{eq:cd-intro} with the rate $O(1/k)$  for general convex objectives. The accelerated coordinate descent method (ACDM) was further developed in \cite{nesterov_efficiency_2012} to achieve a faster rate $O(1/k^2)$. Applying ACDM to \cref{eq:dual-prob} yields the accelerated randomized Kaczmarz (ARK) \cite{lee_efficient_2013,2015An}, which the improved contraction factor $1 - m^{-1/2}/\kappa(A)$. For more related works, we refer to  \cite{fercoq2015accelerated,leventhal_randomized_2010,loizou2020momentum,tondji_acceleration_2024,Wright2015,zeng2024adaptive} and the references therein.

Also, as pointed out in \cite{Griebel2012}, the Kaczmarz method \cref{eq:kacz-intro} is a special instance of the Schwarz iterative methods \cite{griebel_abstract_1995}, which are also known as the subspace correction methods \cite{1992Iterative,xu-2001,xu_method_2002}.
Under such a framework, Oswald and Zhou \cite{oswald_convergence_2015} proved that the cyclic Kaczmarz method converges linearly with a factor of $1 - c/\kappa^2(D^{-1/2}A)$, where $c$ is a mild logarithmic dependence on $m$.

%\subsection{Motivation}
%Nearly singular linear systems are ubiquitous in theory and applications. In CT reconstruction, the performance of  ART, which is based on the Kaczmarz method, is highly dependent on the condition of the system matrix. Studies have shown that when CT projection data suffers from sparse sampling, geometric inaccuracies, or physical effects (such as beam hardening), the system can exhibit ill-conditioned or singular properties, leading to issues in the Kaczmarz method such as slow convergence, solution instability, or oscillatory behavior \cite{beister_iterative_2012}.
%Key examples include $H(\text{grad})$, $H(\text{div})$, and $H(\text{curl})$ 
%finite element discretizations \cite{2002Finite}, and stable discretizations of the nearly incompressible linear elasticity problems \cite{2000Robust,schoberl_multigrid_1999,YI2006A}. 

\subsection{Motivation}
Although Kaczmarz methods are effective for large-scale noisy or inconsistent systems \cite{zouzias_randomized_2013} and there are also parallel and sparse variants \cite{moorman2021randomized,schopfer_linear_2019,Yuan2021,Yuan2021a}, the rate degenerates for ill-conditioned problems. Such ill-condition systems arise in practical applications like CT reconstruction \cite{beister_iterative_2012} and finite element discretizations for $H(\text{grad})$, $H(\text{div})$, and $H(\text{curl})$ systems \cite{2002Finite} and  stable discretizations of the nearly incompressible elasticity problems \cite{YI2006A}. The theoretical rates depend on the scaled condition number $\kappa(A)$ which can be huge and slow down the convergence, especially for {\it nearly singular systems} having small nonzero singular values; see also \cite{steinerberger_randomized_2021}.  

Instead of \cref{Ax=b}, throughout, let us consider the following linear system
\begin{equation}\label{eq:Ax=b}
	A(\epsilon)x = b,
\end{equation}
where $A(\epsilon) \in \mathbb{R}^{m \times n}$ depends on a small parameter $\epsilon > 0$, and $b\in\ran{A(\epsilon)}$ ensures the consistency. We are interested in the limit case $\epsilon \to 0+$, which makes $A(\epsilon)$ ill-conditioned as $\sigma_{\min}^+(A(\epsilon))\to 0$. Look at a simple example
\begin{equation}\label{eq:Ae1}
	A(\epsilon) =
	\begin{bmatrix}
		1 & -1 \\
		1+\epsilon & -1+\epsilon
	\end{bmatrix}.
\end{equation}
The two rows are almost identical and $\lambda_{\min}^+(A^\top(\epsilon)A(\epsilon))=2+\epsilon^2-\sqrt{\epsilon^4+4}=O(\epsilon^2)$.
As shown in \cref{table-comparison}, for this instance, decreasing $\epsilon$ increases $\kappa(A(\epsilon))$ dramatically, and the Kaczmarz method \cref{eq:kacz-intro} becomes inefficient as the two hyperplanes with respect to the rows of $A(\epsilon)$ approach to each other. The heavy zigzag behavior of Kaczmarz can be observed from \cref{simple_zigzag}. This leads to our main motivation: \textit{How to preserve the efficiency of Kaczmarz methods for nearly singular linear systems?}  
\renewcommand{\arraystretch}{1.2}
\begin{table}[H]
	\centering
	\caption{The condition number \(\kappa(A(\epsilon))\) and iterations of Kaczmarz  \cref{eq:kacz-intro} and 		kernel-augmented Kaczmarz \cref{eq:kak} under the stopping criterion: \(\|A(\epsilon)x-b\|/\|b\| < 10^{-7}\).}
	\label{table-comparison}
	\begin{tabular}{@{}ccccc@{}}
		\toprule
		\(\epsilon\) & \(1/5\) & \(1/5^{2}\) & \(1/5^{3}\) & \(1/5^{4}\) \\
		\midrule
		\(\kappa(A(\epsilon))\) & 10 & 50 & 250 & $1.25\times 10^3$ \\
		%		\midrule
		\hline
		Kaczmarz& $4.1\times 10^2$ & $1.0\times 10^4$ & $2.5\times 10^5$ & $6.3\times 10^6$ \\
		%		\midrule
		\hline
		KaK & 16 & 16 & 16 & 16 \\
		\bottomrule
	\end{tabular}
\end{table}
%\kk{这个表格下里的，Kacz取的是1，Kak取的是理论的临界值,没有乘0.9}

\subsection{Main contributions}
So far, robust Kaczmarz methods for general non-square nearly singular systems have not been considered before in the literature. In this work, we extend the classical Kaczmarz \cref{eq:kacz-intro} to the nearly singular problem \cref{eq:Ax=b} and establish the uniformly rate of convergence for small $\epsilon>0$.  More precisely, our main contributions are summarized as below.
\begin{itemize}
	\item We introduce the concept of {\it near singularity} for matrices with unstable row spaces and provide detailed characterizations via the canonical angles in the Grassman manifold. As a generalization of the standard kernel space, we also define a related important set called the {\it approximate kernel}, which is crucial for developing our robust Kaczmarz methods.
	\item Based on the subspace correction presentation \cite{Griebel2012,oswald_convergence_2015} of Kaczmarz, we propose an variant called the kernel-augmented Kaczmarz (KaK), which adopts a stable row space decomposition by using the approximate kernel and possesses a robust linear rate $O((1-\rho)^k)$, where $\rho\in(0,1)$ is uniformly independent on the small parameter $\epsilon$. We refer to	\cref{table-comparison} for numerical evidence on the simple example \cref{eq:Ae1}. 
	\item Using the relation between Kaczmarz and CD,  we derive an equivalent and implementable version  of KaK called the kernel-augmented CD (KaCD). Also, we combine KaCD with a predictor-corrector scheme  \cite[Eq.(83)]{Luo2022} yielding the kernel-augmented accelerated coordinate descent (KaACD). The uniform faster rate $O(\min\left\{1/k^2,\,(1-\sqrt{\rho})^k\right\})$ shall be proved theoretically and verified numerically.
\end{itemize} 
To the best of our knowledge, KaK is the first robust Kaczmarz with a stable rate of convergence for nearly singular systems, and KaACD is the first accelerated coordinate descent with robust linear rate, compared with existing acceleration works \cite{lee_efficient_2013,2015An}.
\subsection{Outline}
The remainder of the paper is organized as follows. In \cref{sec2} we prepare some preliminaries including basic notations, the Grassman manifold and the subspace correction framework. In \cref{sec:nsp-appker} we introduce the near singularity and  the approximate kernel. Then in \cref{sec:ssc-kacz} we give our robust Kaczmarz
method and establish the proof of convergence rate. The equivalent KaCD and its acceleration are presented in \cref{sec:rcd-acc}. Numerical experiments are provided in \cref{sec:experiments} and finally, \cref{conclu} gives some concluding remarks.

\section{Preliminaries}
\label{sec2}
In this section, we prepare some preliminaries including basic notations, the Grassman manifold and the subspace correction framework.
\subsection{Notation}
Let $V$ be a finite dimensional Hilbert space with the inner product $(\cdot,\cdot)$ and the induced norm $\nm{\cdot} = \sqrt{(\cdot,\cdot)}$. Denote by $\mathcal L(V)$ the set of all linear mappings from $V$ to $V$. The identity mapping $I:V\to V$ is $Iv=v$ for all $v\in V$. If $W\subset V$ is a subspace, then define the inclusion operator $\iota_{W\to V}:W\to V$ by the restriction of the identity operator on $W$, i.e., $\iota_{W\to V} = I_W$.

Given any $A\in\mathcal L(V)$, the operator norm is
\[
\nm{A}_{V\to V}: = \sup_{\nm{v}=1}\nm{Av},
\]
and the adjoint operator $A':V\to V$ is defined by
\begin{equation}\label{eq:adj}
	\inner{A'u,v}=	\inner{u,Av}
	\quad\forall\,u,\,v\in V.
\end{equation}
When $A$ is symmetric and positive definite (SPD), i.e., $A=A'$ and $\dual{Av,v}>0$ for all $v\in V\backslash\{0\}$, we also define the $A$-inner product 
\begin{equation}\label{eq:A-inner}
	\inner{u,v}_A:=\inner{Au,v}\quad\forall\,u,v\in V,
\end{equation}
and the induced $A$-norm $\nm{\cdot}_A:=\sqrt{\inner{\cdot,\cdot}_A}$. For any $B\in\mathcal L(V)$, its $A$-norm is defined by
\[
\nm{B}_A :=\sup_{\nm{v}_A=1}\nm{Bv}_A,
\] 
and we denote by $B^t:V\to V$ the adjoint operator of $B$ with respect to (w.r.t.) the $A$-inner product:
\begin{equation}\label{eq:adj-A}
	\innerA{B^tu,v}=	\innerA{u,Bv}
	\quad\forall\,u,\,v\in V.
\end{equation} 
Clearly, for any $B\in\mathcal L(V)$, we have $(BA)^t=B'A$. Particularly, if $B=B'$ is symmetric, then $(BA)^t = BA$, which means $BA$ is symmetric w.r.t the $A$-inner product.

As usual, let $\nm{\cdot}$ be the standard 2-norm of vectors/matrices and $\dual{\cdot,\cdot}$ the Euclidean inner product between vectors. For a linear subspace \( S\subset \R^n \), \(\dim(S)\) refers to its dimension. For any \( C \in \mathbb{R}^{m \times n} \), denote by \( C^\top/C^+ \) the transpose/ Moore-Penrose pseudoinverse of \( C \) and let $\rank C,\,	 \myker{C}$ and $ \ran{C}$ be the rank, the kernel space and the range (or column space) of $C$, respectively. The range of $C^\top$ is called the row space of \( C \) and usually denoted as \row{$C$}. A fundamental fact is that $\row{C} = \row{C^\top C}$ for all $C\in\R^{m\times n}$, which implies $\dim(\row{C})=\rank(CC^\top)$. For a symmetric matrix $C\in\R^{n\times n}$, let \(\lambda_{\min}(C)\) and \(\lambda_{\max}(C)\) be the smallest and largest eigenvalues of \( C \). If \( C \) is symmetric positive semidefinite, then denote by \(\lambda_{\min}^+(C)\) the smallest nonzero eigenvalue of \( C \), namely,
\begin{equation}\label{eq:lambda-min+}
	\lambda_{\min}^+(C):=\min_{v\in\ran{C}\backslash\{0\}} \dual{Cv,v}/\nm{v}^2.
\end{equation}
\subsection{The Grassman manifold}
Let $n,p\in\mathbb N_+$ be such that $p\leq n$. The Grassman manifold $\gras{p,n}$ consists of all $p$-dimensional subspaces of $\R^n$. For any nonzero vectors $u,w\in\R^n\backslash\{0\}$, the angle between $u$ and $w$ is
\[
\angle(u,w):=\arccos\frac{\dual{u,w}}{\|u\|\|w\|}\in[0,\pi].
\]
As an generalization, for any $u\in\R^n\backslash\{0\}$ and $W\in\gras{p,n}$, the angle between $u$ and $W$ is defined by 
\begin{equation}\label{eq:angle-u-W}
	\angle(u,W): =\min_{w\in W\backslash\{0\}}\arccos\frac{\snm{\dual{u,w}}}{\|u\|\|w\|}\in[0,\pi/2].
	%	\left[0,\frac{\pi}{2}\right],
\end{equation}
Let $P_W:\R^n\to W$ be the orthogonal projection, then $\angle(u,W) = \arccos \|P_Wu\|/\|u\|$. 
\begin{Def}[Canonical angles]\label{def:c-angle}
	For any $  U, W\in\gras{p,n}$, assume the columns of $\mathcal U\in\R^{n\times p}$ and $\mathcal W\in\R^{n\times p}$ form orthonormal bases of $U$ and $W$. Let  $0\leq \sigma_1\leq\cdots\leq\sigma_p $ be the singular values of $\mathcal U^\top\mathcal W$. For $1\leq i\leq p$, define $\theta_i:=\arccos\sigma_i\in[0,\pi/2]$. We call $0\leq\theta_p\leq\cdots\leq\theta_1\leq\pi/2$ the canonical angles between $U$ and $W$. The largest canonical angle $\theta_1$ satisfies
	\[
	\cos\theta_1 = \min_{u\in U\backslash\{0\}}\max_{w\in W\backslash\{0\}} \frac{\snm{\dual{u,w}}}{\nm{u}\nm{w}}.
	\]
\end{Def}
\begin{rem}
	For given $  U, W\in\gras{p,n}$, the canonical angles $\{\theta_i\}_{i=1}^p$ is independent on the choices of the orthonormal bases of $U$ and $W$. In particular, we can take $\mathcal U=\left[u_1,\cdots,u_p\right]$ and $\mathcal W=\left[w_1,\cdots,w_p\right]$ such that $\dual{u_i,u_j}=\dual{w_i,w_j}=\delta_{ij}$ and  $\dual{u_i,w_j}=\delta_{ij}\cos\theta_i$ for all $1\leq i,\,j\leq p$.
\end{rem}
\begin{Def}[Gap metric on $\gras{p,n}$]\label{def:gap}
	For any $  U, W\in\gras{p,n}$, define the gap metric by that $\Delta(U,W):=\nm{P_U-P_W}$, where $P_U:\R^n\to U$ and $P_W:\R^n\to W$ are orthogonal projection operators.
\end{Def}

It is well-known that the Grassman manifold $\gras{p,n}$ equipped with the gap metric $\Delta(\cdot,\cdot)$ yields a complete metric space. For any $  U, W\in\gras{p,n}$,  according to \cite[Theorem 4.5 and Corollary 4.6]{Stewart1990}, we have 
\begin{equation}\label{eq:sin-id}
	\Delta(U,W) = \sin\theta_1,
\end{equation} 
where $\theta_1\in[0,\pi/2]$ denotes the largest canonical angle between $U$ and $W$. In particular, we have $\angle(u,\row{v^\top}) = \angle(\row{u^\top}, \row{v^\top})$ for any $u,\,v\in\R^n\backslash\{0\}$.
\begin{lem}\label{lem:lim-Gpn}
	Let $A:[0,\delta]\to\R^{m\times n}$ be a matrix-valued function that is continuous in component-wise. If $\row{A(\epsilon)}\in\gras{r,n}$ for all $\epsilon\in[0,\delta]$, then
	\[
	\lim\limits_{\epsilon\to0+}\Delta( \row{A(\epsilon)},\row{A(0)})=0.
	\]
\end{lem}
\begin{proof}
	The orthogonal projection operator $P_{{A}(\epsilon )}:\R^n\to \row{A(\epsilon)}$ is given by
	\[
	P_{{A}(\epsilon )}=A^\top(\epsilon )\bigl(A(\epsilon )A^\top(\epsilon )\bigr)^\dagger A(\epsilon ),
	\]
	where $(\cdot)^\dagger$ denotes the Moore--Penrose pseudo‑inverse. Since $A(\epsilon)$ is continuous in component-wise and $\row{A(\epsilon)}\in\gras{r,n}$ for all $\epsilon\in[0,\delta]$, we claim that $A(\epsilon )A^\top(\epsilon )$ is also continuous and $\rank{A(\epsilon )A^\top(\epsilon )} =\dim \row{A(\epsilon)}= r$ for all $\epsilon\in[0,\delta]$. Thanks to \cite[Theorem 3.9]{Stewart1990}, $\bigl(A(\epsilon )A^\top(\epsilon )\bigr)^\dagger$  is continuous and so is $P_{{A}(\epsilon )}$. In particular, by 	\cref{def:gap}, we have
	\[
	\lim\limits_{\epsilon\to0+}\Delta( \row{A(\epsilon)},\row{A(0)})=
	\lim_{\epsilon \to0+}\|P_{{A}(\epsilon )}-P_{{A}(0)}\|=0.
	\]
	This completes the proof of this lemma.
\end{proof} 
\begin{lem}\label{lem:angle}
	Let $u,v\in\R^n\backslash\{0\}$ and $W\in\gras{p,n}$ be such that $\angle(u,W)\leq \beta$ and $\angle(v,W)\geq\theta$, with $0\leq\beta\leq \theta\leq\pi/2$. Then $\cos\angle(u,v)\leq \cos(\theta-\beta)$.
\end{lem}
\begin{proof}
	With out loss of generality, assume $u$ and $v$ are unit vectors. Then $\cos\angle(u,W)=\nm{P_Wu}$ and $\cos\angle(v,W)=\nm{P_Wv}$, and we have the orthogonal decomposition: $u=P_Wu+u_0$ and $v=P_Wv+v_0$, where both $u_0=u-P_Wu$ and $v_0=v-P_Wv$ are orthogonal to $W$. Observing $\sin\angle(u,W) = \nm{u_0}$ and $\sin\angle(v,W) = \nm{v_0}$, a direct calculation gives 
	\[
	\begin{aligned}
		{}&		\cos\angle(u,v) = {\dual{u,v}}=  {\dual{P_Wu,P_Wv}+\dual{u_0,v_0}}\\
		={}&\nm{P_Wu}\nm{P_Wv}\cos\angle(P_Wu,P_Wv)+\nm{u_0}\nm{v_0}\cos\angle(u_0,v_0)\\		
		={}&\cos\angle(u,W)\cos\angle(v,W)\cos\angle(P_Wu,P_Wv)+
		\sin\angle(u,W)\sin\angle(v,W)\cos\angle(u_0,v_0)\\
		\leq {}&\cos\angle(u,W)\cos\angle(v,W)+
		\sin\angle(u,W)\sin\angle(v,W) \\
		={}&\cos(\angle(u,W)-\angle(v,W))=\cos(\angle(v,W)-\angle(u,W)).
	\end{aligned}
	\]
	Since $-\angle(u,W)\geq-\beta$ and $\angle(v,W)\geq\theta$, we obtain $\angle(v,W)-\angle(u,W)\geq \theta-\beta\geq0$. It follows immediately that $\cos\angle(u,v)\leq \cos(\theta-\beta)$. This completes the proof.
\end{proof}
\subsection{Successive subspace correction}
\label{sec:ssc}
Let $V$ be a finite dimensional Hilbert space with the inner product $(\cdot,\cdot)$ and the induced norm $\nm{\cdot} = \sqrt{(\cdot,\cdot)}$.  Given an SPD linear operator $A:V\to V$ and $f\in V$, we aim to seek $u\in V$ such that 				
\begin{equation}\label{eq:Auf-msc}
	Au = f.
\end{equation}
Let us restate briefly the basic framework of the  {\it successive subspace correction} (SSC) for solving \cref{eq:Auf-msc}; see \cite{1992Iterative,xu-2001,xu_method_2002} for a comprehensive presentation. Such a framework is crucial for us to analyze the Kaczmarz iteration and develop robust variants for solving nearly singular system.
\subsubsection{Algorithm presentation}
Consider a finite sequence of finite dimensional {\it auxiliary spaces} $\{V_j\}_{j=1}^{J}$. Each $V_j$ is not necessarily a subspace of $V$ but is related to $V$ via a linear operator $R_j:V_j\to V$ that satisfies the following decomposition assumption.
\begin{assum}\label{assum:V-Vj}
	It holds that 
	$ V = \sum_{j=1}^{J}R_jV_j$.
	%	\begin{equation}\label{eq:Vj-V}
		%		V = \sum_{j=1}^{J}R_jV_j.
		%	\end{equation} 
\end{assum}

Thanks to \cref{assum:V-Vj}, for any $v\in V$, we have $		v = \sum_{j=1}^{J}R_jv_j$ with $v_j\in V_j,\,1\leq  j\leq  J$.
Here we do not require such a decomposition is a direct sum, i.e., the representation of $v$ has not to be unique. 

For each auxiliary space, we also propose the following assumption.
\begin{assum}\label{assum:aj-Vj} 		 
	For $1\leq j\leq J$, $V_j$ is a Hilbert space with an inner product $a_j(\cdot,\cdot)$.   
\end{assum}

Let $R_j^*:V\to V_j$ be the adjoint operator of $R_j$, namely,
\[
a_j\inner{R^*_j v,v_j} = \inner{v,R_j v_j}\quad \forall\,v_j\in V_j,\,v\in V.
\] 
Define the linear operator $T_j:V\to V_j$ by that
\[
a_j\inner{T_jv,v_j} = \innerA{v,R_jv_j}
\quad \forall\,v_j\in V_j,\,v\in V.
\]
Clearly, $T_j$ is the adjoint operator of $R_j$ w.r.t. the $A$-inner product and $T_j=R_j^*A$. Note that $R_jR_j^*:V\to V$ is symmetric and thus $R_jT_j=R_jR_j^*A$ is symmetric w.r.t. the $A$-inner product.
%The adjoint operator of $R_j$ w.r.t. the $A$-inner product is given by 
%\[
%a_j\inner{R^{\top}_j v,v_j} = \innerA{v,R_j v_j}\quad \forall\,v_j\in V_j,\,v\in V.
%\]
%By definition, we have $T_j=R_j^*A$.

The method of SSC with relaxation reads as follows (cf.\cite[Algorithm 3.5]{1992Iterative}). Given $u_k\in V$, set $v_{k,1} = u_k$ and update $u_{k+1}=v_{k,J+1}$ by
\begin{equation}\label{eq:vkj}
	v_{k,j+1} = v_{k,j} +\omega R_jR_j^*(f-Av_{k,j})\quad\text{for} \,j=1,\cdots,J,
\end{equation}
where $\omega>0$ denotes the relaxation parameter.
It can be reformulated as a multiplicative Schwarz iteration
\begin{equation}\label{eq:iter-Bssc}
	u_{k+1} = u_k + B_{\rm ssc}(\omega)(f-Au_k),
\end{equation}
with the iterator $B_{\rm ssc}(\omega):V\to V$ satisfies $	I-B_{\rm ssc}(\omega)A =(I-\omega R_JT_J)\cdots(I-\omega R_1T_1)$. 
%In Line \ref{algo:rj} of \cref{algo:SSC}, the residual correction is actually given by $r_j=R_j^*(f-Av)$. Therefore, the update in Line \ref{algo:v} is equivalent to $v=v+R_jR_j^*(f-Av)$.
%%\begin{equation}\label{eq:I-BA-SSC}
%%	I-B_{\rm ssc}(\omega)A =(I-R_JT_J)\cdots(I-R_1T_1).
%%\end{equation}
%%%%%%%%%%%%%%%%%%%%Rj*%%%%%%
%%\LH{A direct calculation gives $(AR_j)^*=R_j^*A$. } 
%%\LH{It is clear that $R_j^{\top} =(AR_j)^*= R^*_jA$.}
%%%%%%%%%%%%%%%%%%%%Rj*%%%%%%
%\begin{algorithm}[H]
%	\caption{SSC}
%	\label{algo:SSC}
%	\begin{algorithmic}[1] 
	%		\STATE Given $u_0\in V$.
	%		\FOR{$k = 0,1,\ldots,$}
	%		\STATE Set $v = u_k$.
	%		\FOR{$j = 1,2,\ldots,J$}
	%		\STATE Find $r_j\in V_j$ such that $
	%		a_j(r_j,w) = \inner{f-Av,R_jw}\quad\forall\,w\in V_j$. \label{algo:rj}
	%		\STATE Update $v= v + R_jr_j$. \label{algo:v} % $=v+R_jR_j^*(f-Av)$.
	%		\ENDFOR
	%		\STATE Update $u_{k+1} = v$.
	%		\ENDFOR
	%	\end{algorithmic}
%\end{algorithm} 
\subsubsection{Symmetrization}
\label{sec:sym-ssc}
Note that $B_{\rm ssc}(\omega)$ is possibly not symmetric and the symmetrized variant of \cref{eq:iter-Bssc} is useful for both theoretical analysis and practical performance. Specifically, based on \cref{eq:iter-Bssc}, consider 
\[
\left\{
\begin{aligned}
	{}&	u_{k+1/2} = u_k + B_{\rm ssc}(\omega)(f-Au_k),\\
	{}& 	u_{k+1} =u_{k+1/2} + B'_{\rm ssc}(\omega)(f-Au_{k+1/2}),
\end{aligned}
\right.
\]
where $B'_{\rm ssc}(\omega):V\to V$ denotes the adjoint operator of $B_{\rm ssc}(\omega)$; see \cref{eq:adj}. This yields the symmetrized SSC \cite[Algorithm 3.4]{1992Iterative}. Given $u_k\in V$, update $u_{k+1}\in V$ by the following two steps:
\begin{itemize}
	\item {\bf Step 1} Set $v_{k,1} = u_k$ and update $u_{k+1/2}=v_{k,J+1}$ by
	\[
	\begin{aligned}
		v_{k,j+1} = v_{k,j} +\omega R_jR_j^*(f-Av_{k,j})\quad\text{for} \,j=1,\cdots,J.
	\end{aligned}
	\]
	\item {\bf Step 2} Set $v_{k,J+1} = u_{k+1/2}$ and update $u_{k+1} = v_{k,1}$ by that
	\[
	v_{k,j} = v_{k,j+1} +\omega R_jR_j^*(f-Av_{k,j+1})\quad\text{for} \,j=J,\cdots,1.
	\]
\end{itemize}
Actually, this is also equivalent to a multiplicative Schwarz iteration
\begin{equation}\label{eq:iter-bar-Bssc}
	u_{k+1} = u_k + \bar B_{\rm ssc}(\omega)(f-Au_k),
\end{equation}
where $\bar B_{\rm ssc}(\omega) = B'_{\rm ssc}(\omega)+B_{\rm ssc}(\omega)-B'_{\rm ssc}(\omega)AB_{\rm ssc}(\omega)$ is called the symmetrization of $B_{\rm ssc}(\omega)$. Similarly as before, a direct computation gives 
\begin{equation}\label{eq:I-barBA-I-BA-SSC}
	I-\bar B_{\rm ssc}(\omega)A = (I-B'_{\rm ssc}(\omega)A)(I-B_{\rm ssc}(\omega)A)
	=(I-B_{\rm ssc}(\omega)A)^t(I-B_{\rm ssc}(\omega)A),
\end{equation}
which is symmetric w.r.t. the $A$-inner product. 
%
%version of \cref{algo:SSC}; see \cref{algo:Sym-SSC} below. 
%\begin{algorithm}[H]
%	\caption{Symmetrized SSC}
%	\label{algo:Sym-SSC}
%	\begin{algorithmic}[1] 
	%		\STATE Given $u_0\in V$.
	%		\FOR{$k = 0,1,\ldots,$}
	%		\STATE Set $v = u_k$.
	%		\FOR{$j = 1,2,\ldots,J$}
	%%		\STATE Find $r_j\in V_j$ such that $
	%%		a_j(r_j,w) = \inner{f-Av,R_jw}\quad\forall\,w\in V_j$.
	%		\STATE Update $v =v+R_jR_j^*(f-Av)$.
	%		\ENDFOR
	%		\FOR{$j = J,J-1,\ldots,1$}
	%%		\STATE Find $r_j\in V_j$ such that $
	%%		a_j(r_j,w) = \inner{f-Av,R_jw}\quad\forall\,w\in V_j$.
	%		\STATE Update $v=  v+R_jR_j^*(f-Av)$.		 
	%		\ENDFOR		
	%		\STATE Update $u_{k+1} = v$.
	%		\ENDFOR
	%	\end{algorithmic}
%\end{algorithm}

\subsubsection{Convergence theory}

Based on \cref{assum:V-Vj}, we introduce an auxiliary space of
product type $\bm V := V_1\times  V_2\times \cdots\times V_J$ and the corresponding inner product $\inner{\bm u,\bm v}_{\bm V} := \sum_{j=1}^{J}a_j\inner{u_j,v_j}$ for all $\bm u,\,\bm v\in\bm V$ with $\bm{u}_j = u_j$ and $\bm v_j = v_j$. The induced norm is $\nm{\cdot}_{\bm V}:=\sqrt{\inner{\cdot,\cdot}_{\bm V}}$. Define the linear operator $\mathcal R:\bm V\to V$ by that
\[
\mathcal R\bm u = \sum_{j=1}^{J}R_j\bm u_j=\sum_{j=1}^{J}R_ju_j.
\]
We claim that $\mathcal{R}:\bm V\to V$ is onto, i.e., $\mathcal{R}$ is surjective. Denote by $\mathcal{R}^\star:V\to \bm V$ the adjoint operator of $\mathcal R$:
\[
\inner{\mathcal R^\star u,\bm v}_{\bm V}=\inner{u,\mathcal R\bm v}.
\] 
We also define the adjoint operator of $\mathcal R$ w.r.t. the $A$-inner product:
\begin{equation}\label{eq:RT}
	\inner{\mathcal R^{T} u,\bm v}_{\bm V}=\innerA{u,\mathcal R\bm v}.
\end{equation}
It is clear that $\mathcal R^{T} =(A\mathcal R)^\star= \mathcal R^\star A$. 
%\LH{Shall we use $()^T$ for operators?}

Let $\bm A = \mathcal{R}^\star A\mathcal{R}$ be the expanded operator. If we write each $\bm v\in\bm V$ as a ``column vector", then in a consistent block form, we have $\mathcal R=(R_1,\cdots,R_J)$ and 
\[
\bm A = \begin{bmatrix}
	T_1 R_1&	T_1 R_2&\cdots&	T_1 R_J\\
	T_2 R_1&	T_2 R_2&\cdots&		T_2 R_J\\
	\vdots&	\vdots&\ddots&	\vdots\\
	T_J R_1&	T_J R_2&\cdots&		T_J R_J
\end{bmatrix},\quad \mathcal{R}^\star = \begin{bmatrix}
	R_1^*\\R_2^*\\\vdots\\R_J^*
\end{bmatrix},\quad  \bm v = \begin{bmatrix}
	\bm v_1\\\bm v_2\\\vdots\\\bm v_J
\end{bmatrix}.
\]
Now consider the block splitting $		\bm A = \bm D +\bm L +\bm U$, where $\bm D =\diag{\bm A}$ denotes the block diagonal part and $\bm U=\triu{\bm A}$ is the strict upper block triangular part. It is not hard to verify that $\bm L'=\bm U$ and $\bm D'=\bm D$.

Under the following assumption, which is equivalent to that $2\bm I-\omega\bm D$ is SPD, we can establish the convergence results of \cref{eq:iter-Bssc,eq:iter-bar-Bssc}. It is known as the famous Xu--Zikatanov identity \cite{xu_method_2002} and provides a powerful tool for the analysis of the Kaczmarz iteration and its variants. 
%For a complete proof of \cref{thm:xz-id}, we refer to \cref{sec:app-xz-id}.
\begin{assum}\label{assum:TjRj} 
	$2\bm I-\omega\bm D$ is SPD. Or equivalently, $\omega\in\left(0,2/\nm{\bm D}_{\bm V\to \bm V}\right)$.
\end{assum} 
\begin{thm}[\cite{xu_method_2002}]
	\label{thm:xz-id}
	Under Assumptions \ref{assum:V-Vj}, \ref{assum:aj-Vj} and \ref{assum:TjRj}, both \cref{eq:iter-Bssc,eq:iter-bar-Bssc} are convergent and 					 
	\begin{equation}\label{eq:xz-id}
		\tag{XZ-Identity}
		\nm{I-\bar{B}_{\rm ssc}(\omega)A}_A=	\nm{I-B_{\rm ssc}(\omega)A}_A^2 =  1-\frac{1}{c_0(\omega)},
	\end{equation} 
	where the finite positive constant $c_0(\omega)$ is defined by
	\begin{align}
		\label{eq:c0}
		c_0(\omega): ={} \frac{1}{\omega}\sup_{\nm{v}_A=1} \inf_{\mathcal{R}\bm v =v}\inner{	\left(2\bm I-\omega\bm D\right)^{-1}	\left(\bm I+\omega\bm U\right)\bm v,	\left(\bm I+\omega\bm U\right)\bm v}_{\bm V}
	\end{align} 
\end{thm}

As an important byproduct, we have the following corollary.
\begin{coro}\label{coro2.1}
	Under Assumptions \ref{assum:V-Vj}, \ref{assum:aj-Vj} and \ref{assum:TjRj}, for the symmetrized iterator $\bar{B}_{\rm ssc}(\omega)$, we have 
	\begin{equation}\label{eq:scv-barB}
		\frac1{c_0}\dual{ \bar{B}_{\rm ssc}^{-1}(\omega)v,v}\leq 		\dual{Av,v}\leq \dual{ \bar{B}_{\rm ssc}^{-1}(\omega)v,v}\quad\forall\,v\in V,
	\end{equation}
	where $c_0\geq c_0(\omega)$ is arbitrary, with $c_0(\omega)>0$ being given by \cref{eq:c0}.
\end{coro}
\begin{proof}
	We refer to \cite[Lemma 2.1]{1992Iterative} and omit the details.
\end{proof}
%\begin{coro}
%	Moreover, if $B_{\rm ssc}(\omega)=B'_{\rm ssc}$, then 
%	\[
%	\nm{I-B_{\rm ssc}(\omega)A}_A=\max\{\snm{1-	\lambda_{\min}({B}_{\rm ssc}A)},\,\snm{1-	\lambda_{\max}({B}_{\rm ssc}A)}\},
%	\]
%	which implies that 
%	\[
%	\begin{aligned}
	%		\snm{1-	\lambda_{\min}({B}_{\rm ssc}A)}\leq \nm{I-\bar{B}_{\rm ssc}A}_A^{1/2}=\sqrt{1-\lambda_{\min}(\bar{B}_{\rm ssc}A)},\\
	%		\snm{1-	\lambda_{\max}({B}_{\rm ssc}A)}\leq \nm{I-\bar{B}_{\rm ssc}A}_A^{1/2}=\sqrt{1-\lambda_{\min}(\bar{B}_{\rm ssc}A)}.
	%	\end{aligned}
%	\]
%	We obtain 
%	\[
%	\begin{aligned}
	%		\lambda_{\min}({B}_{\rm ssc}A)\geq {}& 1-\sqrt{1-\lambda_{\min}(\bar{B}_{\rm ssc}A)}\geq\frac{\lambda_{\min}(\bar{B}_{\rm ssc}A)}{2},\\
	%		\lambda_{\max}({B}_{\rm ssc}A)\leq {}& 1+\sqrt{1-\lambda_{\min}(\bar{B}_{\rm ssc}A)}\leq2,
	%	\end{aligned}
%	\]
%	This gives
%	\begin{equation}\label{eq:scv-B}
	%		\frac{\lambda_{\min}(\bar{B}_{\rm ssc}A)}{2}\dual{ {B}_{\rm ssc}^{-1}v,v}\leq 		\dual{Av,v}\leq 2\dual{ {B}_{\rm ssc}^{-1}v,v}.
	%	\end{equation}
%\end{coro}

\section{Nearly Singularity Property}
\label{sec:nsp-appker}
In \cite{lee_robust_2007}, Lee et al. considered a special nearly singular case with $A(\epsilon)=\epsilon I+A_0$, where $A_0$ is symmetric positive semidefinite. It is clear that $A(\epsilon)$ is SPD for all $\epsilon>0$ but $\lambda_{\min}(A(\epsilon))\to0$ as $\epsilon\to 0$. In this section, we aim to extend such a nearly singular property to a more general non-square case. 
\subsection{Definition}
As we all know, the Kaczmarz iteration is actually a row action method \cite{karczmarz1937angenaherte} and can also be recast into the SSC framework \cite{oswald_convergence_2015,xu_method_2002} with proper subspace decomposition on the row space of $A(\epsilon)$; see later in \cref{sec:ssc-kacz}. Motivated by this, instead of the degeneracy of the smallest eigenvalue or singular value, we now focus on the asymptotic behavior of $\row{A(\epsilon)}$ as $\epsilon$ approaches to 0. More precisely, assume that $A(\epsilon)$ admits the two-block structure as specified below in \cref{assum:A0-A1}.
%\begin{equation}\label{eq:Ae}
%	A(\epsilon) = \begin{bmatrix}
	%		A_{0}(\epsilon) \\
	%		A_{1}(\epsilon)
	%	\end{bmatrix},
%\end{equation}
%where $A_0(\epsilon)\in\R^{m_0\times n}$ and $A_1(\epsilon)\in\R^{m_1\times n}$ with $m_0+m_1=m$. 
\begin{assum}
	\label{assum:A0-A1}
	Let $\bar{\epsilon}>0$. The matrix-valued function $A:[0,\bar{\epsilon}]\to\R^{m\times n}$ is continuous in component-wise and admits the two-block structure
	\begin{equation}\label{eq:Ae}
		A(\epsilon) = \begin{bmatrix}
			A_{0}(\epsilon) \\
			A_{1}(\epsilon)
		\end{bmatrix} \quad \forall\,0\leq \epsilon\leq\bar{\epsilon},
	\end{equation}
	where $A_0(\epsilon)\in\R^{m_0\times n}$ and $A_1(\epsilon)\in\R^{m_1\times n}$ with $m_0+m_1=m$. Moreover, the diagonal part $D_\epsilon:=\diag{A(\epsilon)A^\top(\epsilon)}:[0,\bar{\epsilon}]\to\R^{m\times m}$ is non-degenerate
	\begin{equation}\label{eq:d0-d1}
		\eta_0I \preceq 
		%		\diag{A(\epsilon)A^\top(\epsilon)}
		D_\epsilon
		\preceq \eta_1I\quad\forall\,0\leq \epsilon\leq\bar{\epsilon},
	\end{equation}
	where $0<\eta_0\leq \eta_1<\infty$ are independent on $\epsilon$.
	%	 Moreover, there exists $M_0>0$ such that $\nm{A(\epsilon)-A(0)}\leq M_0\epsilon$ for all
\end{assum}

From the SSC perspective, the stability of the row space plays a key role. Inspired by this, we introduce  the following definition of {\it nearly singular property} that leads to the unstable subspace decomposition. 
\begin{Def}[Nearly singular property]\label{def:nsp}
	Let $A:[0,\bar{\epsilon}]\to\R^{m\times n}$ be a matrix-valued function satisfying \cref{assum:A0-A1}. We say that $A(\epsilon)$ is   nearly singular  at $\epsilon=0$ (or simply nearly singular) if it satisfies:
	\begin{itemize}
		\item[(i)]  $ \dim \row{A_{0}(\epsilon)}<\dim \row{A(\epsilon)}$ for all $0<\epsilon\leq\bar{\epsilon} $;
		\item[(ii)] there exists  $r_0\in\mathbb N_+$ such that   $\dim \row{A_{0}(\epsilon)}=r_0$ for all $0\leq\epsilon\leq\bar{\epsilon}$;		 
		\item[(iii)] there is $\Pi\in\R^{m_1\times m_0}$ such that $A_1(0)=\Pi A_0(0)$. In other words, we have $ \row{A(0)}\subset  \row{A_{0}(0)}$.
	\end{itemize}
\end{Def}

Let us discuss briefly the above definition. The first term says that $\row{A_{0}(\epsilon)}$ is indeed a proper subspace of $\row{A(\epsilon)}$; the second implies that the dimension of  $\row{A_0(\epsilon)}$ leaves invariant for all $\epsilon\in[0,\bar{\epsilon}]$ and is continuous at $\epsilon=0$. Notably, as $\epsilon\to 0+$, $ \row{A_{0}(\epsilon)} $  converges to $   \row{A_{0}(0)}$  under the gap metric. 
\begin{lem}\label{lem:lim-gap-A0}
	Let $A:[0,\bar{\epsilon}]\to\R^{m\times n}$ satisfy \cref{assum:A0-A1}. If $A(\epsilon)$ is nearly singular, then $\row{A_{0}(\epsilon)}\in\gras{r_0,n}$ for all $\epsilon\in[0,\bar{\epsilon}]$ and 
	\[
	\lim\limits_{\epsilon\to0+}\Delta( \row{A_{0}(\epsilon)},\row{A_{0}(0)})=0,
	\]
	where $\Delta(\cdot,\cdot)$ denotes the gap metric (cf.\cref{def:gap}).
\end{lem}
\begin{proof}
	Applying \cref{lem:lim-Gpn} to $A_0:[0,\bar{\epsilon}]\to\R^{m\times n}$ concludes the proof.
\end{proof} 

However, for the whole row space $\row{A(\epsilon)}$, the {\it dimension reduction}  occurs at $\epsilon=0$ due to the inclusion relation in the third term in \cref{def:nsp}. If we consider a  decomposition of $\row{A(\epsilon)}$ by the row vectors, then it is {\it unstable} when $\epsilon\to0+$. For instance, let $A(\epsilon)$ be given by \cref{eq:Ae1}, which has $\row{A(\epsilon)} = \R^2$ for all $\epsilon>0$. With $A_0(\epsilon) = [1,-1]$ and $A_1(\epsilon) = [1+\epsilon,-1+\epsilon]$, it is clear that $\dim \row{A_{0}(\epsilon)}=1 $ for all $\epsilon\geq0$. Hence, by \cref{def:nsp}, $A(\epsilon)$ is nearly singular.
\subsection{Approximate kernel}
We then introduce the {\it approximate kernel} of a nearly singular matrix $A(\epsilon)$, which is crucial for us to develop a robust iterative method from the SSC framework. 
\begin{Def}[Approximate kernel]\label{def:ak}
	Let $A:[0,\bar{\epsilon}]\to\R^{m\times n}$ be a matrix-valued function satisfying \cref{assum:A0-A1}. 
	The approximate kernel of $A(\epsilon)$ is defined by	
	\begin{equation}\small
		\label{eq:app-ker}
		\appker{A(\epsilon)}:=\myker{A_0(\epsilon)}\cap \row{A(\epsilon)}\\= \left\lbrace   A^{\top}(\epsilon)y:A_{0}(\epsilon) A^{\top}(\epsilon)y=0,\,y\in\R^m\right\rbrace.
	\end{equation}
	The approximate dual kernel of $A(\epsilon)$ is given by
	\begin{equation}
		\label{eq:app-dual-ker}
		\dappker{A(\epsilon)}:=\myker{A_0(\epsilon)A^\top(\epsilon)}= \left\lbrace   y\in \R^{m}:A_{0}(\epsilon)A^{\top}(\epsilon)y=0\right\rbrace.
	\end{equation}
\end{Def}

%\begin{rem}
%	Prove that 
%	\begin{itemize}
	%		\item $
	%		\lim\limits_{\epsilon\to0}\row{A_{0}(\epsilon)}= \row{A_{0}(0)}$.
	%		\item 	Define the orthogonal complement of  \(\appker{A(\epsilon)}\) in $V=\row{A(\epsilon)}$ by
	%		\[
	%		\appker{A(\epsilon)}^\perp := \{ v_r \in V : \inner{v_r, v_n } = 0 \quad\forall\, v_n\in \appker{A(\epsilon)}\}.
	%		\]
	%		We claim that \( V = \appker{A(\epsilon)} \oplus \appker{A(\epsilon)}^\perp \) and any \( v \in V \) admits a unique decomposition \( v = v_n + v_r \) where \( v_n \in \appker{A(\epsilon)}\), \( v_r \in \appker{A(\epsilon)}^\perp \) and \( \inner{v_r, v_n } = 0 \). Prove that there exists some positive constant $c_0>0$, independent on $\epsilon$, such that
	%		\[\inf_{\|v\|=1}v^\top A_0(0)^\top  A_0(0)v=\inf_{\|v\|=1}v_r^\top A_0(0)^\top  A_0(0)v_r\geq c_0.\] 
	%		\item $\lambda^+_{\min}( A^\top(\epsilon) A(\epsilon) )\leq C(\epsilon)$, where as $\epsilon\to0$, we have $C(\epsilon)\to0$. 
	%	\end{itemize}
%\end{rem}

We claim that the approximate kernel $\appker{A(\epsilon)}$ is actually the orthogonal complement of $\row{A_0(\epsilon)}$ in $\row{A(\epsilon)}$; see the proof of \cref{lem:lambda-min-A0}. In addition, $\appker{A(\epsilon)}$ is the image of the approximate dual kernel $	\dappker{A(\epsilon)}$ under the mapping $A^\top(\epsilon):\R^m\to\R^n$. It follows that \( \dim \appker{A(\epsilon)} \leq \dim 		\dappker{A(\epsilon)}\). When \( A(\epsilon) \) has full row rank, 
the equality case \( \dim \appker{A(\epsilon)} = \dim 		\dappker{A(\epsilon)}\) holds.

Note that our \cref{def:nsp} does not involve explicitly the degenerate behavior of the singular values of $A(\epsilon)$. In the following, we establish an upper bound of $\lambda^+_{\min}(A^\top(\epsilon) A(\epsilon))$, which goes to zero as $\epsilon\to0+$. 
\begin{lem}\label{lem:up-bd-lambda-min+}
	Let $A:[0,\bar{\epsilon}]\to\R^{m\times n}$ satisfy \cref{assum:A0-A1} and $a^\top_i(\epsilon)$ the $i$-th row of $A_1(\epsilon)$ for $1\leq i\leq m_1$. Let $\theta_1(\epsilon)\in[0,\pi/2]$ denote the largest canonical angle between $\row{A_0(\epsilon)}$ and $\row{A_0(0)}$, and define $\beta(\epsilon): = \max_{1\leq i\leq m_1}\beta_i(\epsilon)$ with $\beta_i(\epsilon)=\angle(a_i(\epsilon),\row{A_1(0)})\in[0,\pi/2]$. If $A(\epsilon)$ is nearly singular, then   
	\begin{equation}
		\label{eq:bound-angle}
		\theta_1(\epsilon)+\beta(\epsilon)\to0\quad\text{as}\,\,\epsilon\to0+,
	\end{equation}
	which implies there exists some $\epsilon_0\in(0,\bar{\epsilon}]$ such that 
	\begin{equation}\label{eq:bd-theta-beta}
		0\leq 		\theta_1(\epsilon)+\beta(\epsilon)\leq \pi/2\quad\forall\,\epsilon\in[0,\epsilon_0].
	\end{equation}
	In addition, we have the estimate
	\begin{equation}\label{eq:lmin+est}
		\lambda_{\min}^+(A^\top(\epsilon )A(\epsilon ))\leq	m_1\eta_1\sin^2\left(\theta_1(\epsilon)+\beta(\epsilon)\right)\quad\forall\,\epsilon\in[0,\epsilon_0].
	\end{equation}
	%Moreover, when $\epsilon\to0+$, we have $\theta_1(\epsilon)+\beta(\epsilon)\to0$.
\end{lem}
\begin{proof}
	Let us firstly verify the statement \cref{eq:bound-angle}. From \cref{eq:sin-id}, we claim that  $\theta_1(\epsilon)=\arcsin\Delta(\row{A_0(\epsilon )},\,\row{A_0(0)})$ is continuous and it follows from \cref{lem:lim-gap-A0} that $\theta_1(\epsilon)\to0$ as $\epsilon\to0+$. For $1\leq i\leq m_1$, $a_i:[0,\bar{\epsilon}]\to\R^n$ is continuous and $\row{a^\top_i(\epsilon)}\in\gras{1,n}$ for all $0\leq \epsilon\leq \bar{\epsilon}$. Then we find that
	\[
	\begin{aligned}
		\beta_i(\epsilon)={}&\angle(a_i(\epsilon),\row{A_1(0)})\leq \angle(a_i(\epsilon),\row{a_i^\top(0)})\\
		={}& \angle(\row{a_i(\epsilon)},\row{a_i^\top(0)})
		=\arcsin\Delta(\row{a^\top_i(\epsilon)},\row{a^\top_i(0)}).
	\end{aligned}
	\]
	Thus, by \cref{lem:lim-Gpn}, we have $\beta_i(\epsilon)\to0$ as $\epsilon\to0+$, which yields that $\beta(\epsilon)\to0$ as $\epsilon\to0+$ and concludes \cref{eq:bound-angle}. It is clear that $\beta(\epsilon)$ is continuous and there exists some $\epsilon_0\in(0,\bar{\epsilon}]$ such that $\theta_1(\epsilon)+\beta(\epsilon)\in[0,\pi/2]$ for all $\epsilon\in[0,\epsilon_0]$. This verifies \cref{eq:bd-theta-beta}.
	
	In what follows, let us prove the estimate \cref{eq:lmin+est}. It is evident that (cf.\cref{eq:lambda-min+})
	\[
	\lambda_{\min}^+(A^\top(\epsilon )A(\epsilon ))=\min_{ v\in\row{A(\epsilon )}\backslash\{0\}} \nm{A(\epsilon )v}^2/\nm{v}^2.
	\]
	Here, we used the fact $\ran{A^\top(\epsilon )A(\epsilon )}=\ran{A^\top(\epsilon )}=\row{A(\epsilon )}$. Since $A(\epsilon)$ is nearly singular, by \cref{def:nsp}, we must have $\dim{\appker{A(\epsilon )}}=\dim{\row{A(\epsilon )}}-r_0>0$ for all $\epsilon\in(0,\bar{\epsilon}]$. Therefore, for any  unit  vector $v(\epsilon )\in\appker{A(\epsilon )}$, it follows 
	\begin{equation}\label{eq:key-est}
		\begin{aligned}
			{}&	\lambda_{\min}^+(A^\top(\epsilon )A(\epsilon ))\leq\nm{A(\epsilon )v(\epsilon)}^2=\|A_1(\epsilon )v(\epsilon )\|^2	=\sum_{i=1}^{m_1}|\dual{a_i(\epsilon ),v(\epsilon )}|^2\\
			\leq {}&\max_{1\leq i\leq m_1}\cos^2\angle(v(\epsilon),a_i(\epsilon)) \cdot\sum_{i=1}^{m_1}\|a_i(\epsilon )\|^2
			\leq  m_1\eta_1\max_{1\leq i\leq m_1}\cos^2\angle(v(\epsilon),a_i(\epsilon)),
		\end{aligned}
	\end{equation}
	where in the last line, we used the assumption \cref{eq:d0-d1}, which promises $\nm{a_i(\epsilon)}^2\leq \eta_1$ for all $1\leq i\leq m_1$.
	It is sufficient to find an upper bound estimate of $\cos\angle(v(\epsilon),a_i(\epsilon))$, which is equivalent to get a lower bound of  $\angle(v(\epsilon),a_i(\epsilon))$ for all $1\leq i\leq m_1$.
	
	For that, we aim to find a lower bound of $\angle(v(\epsilon),\row{A_0(0)})$ and an upper bound of $\angle(a_i(\epsilon),\row{A_0(0)})$. 
	According \cref{def:nsp}, we have $\row{A_1(0)}\subset\row{A_0(0)}$, which indicates that
	\begin{equation}\label{eq:up-bd-angle-ai}
		\angle(a_i(\epsilon),\row{A_0(0)})\leq \angle(a_i(\epsilon),\row{A_1(0)})=\beta_i(\epsilon)\leq \beta(\epsilon).
	\end{equation}
	Then, consider any  unit  vector $u\in\row{A_0(0)}$, which admits the unique decomposition $u=u_1+u_2$ with $u_1\in\row{A_0(\epsilon )}$ and $u_2\in \myker{A_0(\epsilon)}$. It is easy to find that $\nm{u_1}=\cos\angle(u,\row{A_0(\epsilon)}) $ and $\|u_2\| = \sin\angle(u,\row{A_0(\epsilon)})\le\sin\theta_1(\epsilon)$, since $\theta_1(\epsilon)$ is the largest canonical angle between $\row{A_0(\epsilon)}$ and $\row{A_0(0)}$. Observing $v(\epsilon )$ is orthogonal to $u_1$, we also obtain $	|\dual{v(\epsilon ),u}|=|\dual{v(\epsilon ),u_2}|
	\le\|v(\epsilon )\|\,\|u_2\|=\nm{u_2}\le\sin\theta_1(\epsilon)$ for any unit vector $u\in\row{A_0(0)}$.  Thus, we conclude that  
	\begin{equation}\label{eq:lower-bd}
		\angle(v(\epsilon ),\row{A_0(0)})\geq\pi/2-\theta_1(\epsilon).
	\end{equation}
	
	Finally, thanks to \cref{eq:bd-theta-beta,eq:up-bd-angle-ai,eq:lower-bd}, invoking \cref{lem:angle} yields 
	\[
	\cos\angle(v(\epsilon),a_i(\epsilon))\leq\cos(\pi/2-\theta_1(\epsilon)-\beta(\epsilon))=\sin(\theta_1(\epsilon)+\beta(\epsilon )),
	\]
	for all $1\leq i\leq m_1$. This together with \cref{eq:key-est} gives \cref{eq:lmin+est} and completes the proof.
\end{proof}
%%%这里有一个很重要的注解，解释remark2.1的，被我删掉了。%%%%%%%%%%%%
%%%%%%%%%%%%%%%%%%
%%%%%%	To elaborate, let \( \bar{W}_{\epsilon} = \operatorname{span}\{\bar{w}_1, \ldots, \bar{w}_M\} \) with \( \dim(\bar{W}_{\epsilon}) = M \). Then \( W_{\epsilon} = A^{\top}\bar{W}_{\epsilon} = \operatorname{span}\{A^{\top}\bar{w}_1, \ldots, A^{\top}\bar{w}_M\} \). Consider a linear combination \( \sum_{j=1}^M c_j A^{\top}\bar{w}_j = 0 \). Since \( A^{\top} \) has full column rank, this implies \( \sum_{j=1}^M c_j \bar{w}_j = 0 \), and hence \( c_j = 0 \) for all \( j \). Thus, \( \{A^{\top}\bar{w}_j\} \) forms a linearly independent set, ensuring \( \dim(W_{\epsilon}) = M \) when \( A \) has full row rank.
%Then a set of basis vectors for $W_\epsilon$ can be obtained by multiplying the transpose of $A$ with a set of basis vectors for $\bar W_\epsilon$. The  simple examples are provided below:\\

As a comparison, we prove that $\lambda^+_{\min}(A^\top_0(\epsilon) A_0(\epsilon))$ is bounded below by a small perturbation of $\lambda^+_{\min}(A^\top_0(0) A_0(0))$, which is independent of $\epsilon$. This is crucial for proving the robust convergence rates of our proposed methods.
\begin{lem}\label{lem:lambda-min-A0}
	Let $A:[0,\bar{\epsilon}]\to\R^{m\times n}$ be a matrix-valued function satisfying \cref{assum:A0-A1}. If $A(\epsilon)$ is nearly singular, then we have the orthogonal direct sum $\row{A(\epsilon)}=\appker{A(\epsilon)}\oplus\row{A_0(\epsilon)}$, and there exists $\epsilon_1\in(0,\bar{\epsilon}]$ such that
	\begin{equation}\label{eq:sigma0}
		\sigma_0(\epsilon):= \lambda_{\min}^+(A^\top_0(0)A_0(0))\cos^2\theta_1(\epsilon)-\nm{A_0(0)-A_0(\epsilon)}^2>0\quad\forall\,\epsilon\in[0,\epsilon_1],
	\end{equation}
	where $\theta_1(\epsilon)\in[0,\pi/2]$ denotes the largest canonical angle between $\row{A_0(\epsilon)}$ and $\row{A_0(0)}$. Moreover,  we have $	\lambda_{\min}^+(A^\top_0(\epsilon)A_0(\epsilon))\geq \sigma_0(\epsilon)>0$ for all $\epsilon\in[0,\epsilon_1]$ and 
	\begin{equation}\label{eq:lim-sigma0}		\lim\limits_{\epsilon\to0+}\sigma_0(\epsilon) = \lambda_{\min}^+(A^\top_0(0)A_0(0)).
	\end{equation}
\end{lem}
\begin{proof}
	Denote by $\row{A_0(\epsilon)}^\perp$ the orthogonal complement of $\row{A_0(\epsilon)}$ in $\row{A(\epsilon)}$. Let us first show that $\appker{A(\epsilon )}=\row{A_0(\epsilon)}^\perp$.  
	By the definition \cref{eq:app-ker}, $\appker{A(\epsilon )}=\myker {A_0(\epsilon )}\cap\row{A(\epsilon )}$. Since $\row {A_0(\epsilon )}\subset\row{A(\epsilon )}$ and $\row{A_0(\epsilon )}$ is orthogonal to $\myker{A_0(\epsilon )}$, we have $	\appker{A(\epsilon )} \subset \row {A_0(\epsilon )}^\perp$. It follows that 
	\[
	\begin{aligned}
		\dim \appker{A(\epsilon )} \leq \dim  \row {A_0(\epsilon )}^\perp={}&\dim  \row {A(\epsilon )}-\dim  \row {A_0(\epsilon )}\\={}&\dim  \row {A(\epsilon )}-r_0.
	\end{aligned}
	\]
	On the other hand, observe the formula
	\[
	\begin{aligned}
		{}&		\dim \appker{A(\epsilon )} \\
		={}&\dim  \row {A(\epsilon )}+\dim  \myker{A_0(\epsilon )}-\dim(\row {A(\epsilon )}+\myker{A_0(\epsilon )})\\
		\geq{}&\dim  \row {A(\epsilon )}+n-r_0-n=\dim  \row {A(\epsilon )}-r_0,
	\end{aligned}
	\]
	which yields the relation $			\dim \appker{A(\epsilon )} = \dim  \row {A_0(\epsilon )}^\perp$. Hence, we conclude that $\appker{A(\epsilon )}=\row{A_0(\epsilon)}^\perp$ and the orthogonal decomposition $\row{A(\epsilon)}=\appker{A(\epsilon)}\oplus\row{A_0(\epsilon)}$ holds.
	
	According to \cref{eq:sin-id,lem:lim-gap-A0}, we claim that   $\theta_1(\epsilon)\to0$ as $\epsilon\to0+$, and by \cref{assum:A0-A1}, $\nm{A_0(0)-A_0(\epsilon)}\to0$ as $\epsilon\to0+$. This concludes \cref{eq:sigma0,eq:lim-sigma0}.
	
	It remain to verify $\lambda_{\min}^+(A^\top_0(\epsilon)A_0(\epsilon))\geq \sigma_0(\epsilon)$. For any $v\in\row{A_0(\epsilon)}$, we have the orthogonal decomposition $v=v_n+v_r$, where $v_n\in\myker{A_0(0)}$ and $v_r\in\row{A_0(0)}$. Since $\theta_1(\epsilon)\in[0,\pi/2]$ denotes the largest canonical angle between $\row{A_0(\epsilon)}$ and $\row{A_0(0)}$, we also get $	\nm{v_n} = \nm{v}\sin\angle(v,\row{A_0(0)})\leq\nm{v}\sin\theta_1(\epsilon)$, implying $\nm{v_r}^2=\nm{v}^2-\nm{v_n}^2\geq(1-\sin^2\theta_1(\epsilon))\nm{v}^2=\cos^2\theta_1(\epsilon)\nm{v}^2$. Thus, it follows that
	\[
	\small
	\begin{aligned}
		\nm{A_0(\epsilon)v}^2\geq {}&	\frac{1}{2}\nm{A_0(0)v}^2-\nm{(A_0(0)-A_0(\epsilon))v}^2
		={}\frac{1}{2}\nm{A_0(0)v_r}^2-\nm{(A_0(0)-A_0(\epsilon))v}^2\\
		\geq{}& \lambda_{\min}^+(A^\top_0(0)A_0(0))\nm{v_r}^2-\nm{A_0(0)-A_0(\epsilon)}^2\nm{v}^2
		\geq{}\sigma_0(\epsilon)\nm{v}^2.
	\end{aligned}
	\] 
	This leads to $\lambda_{\min}^+(A^\top_0(\epsilon)A_0(\epsilon))\geq \sigma_0(\epsilon)$ and completes the proof.
\end{proof}

To the end of this section, we provide three examples for further illustrations. 
\begin{eg}\label{eg:ex1}
	Let $A(\epsilon)$ be given by \cref{eq:Ae1}, which is nearly singular with $A_0(\epsilon)=[1,-1]$ and $A_1(\epsilon)=[1+\epsilon,-1+\epsilon]$. A direct calculation gives $A_{0}(\epsilon)A^{\top}(\epsilon)=[2,2]$, which yields $	\dappker{A(\epsilon)}=\myspan\{[1,-1]^\top\}$ and $\appker{A(\epsilon)}=A^\top(\epsilon)\dappker{A(\epsilon)}=\myspan\{[1,1]^\top\}$. Besides, we have $\lambda_{\min}^+(A^\top_0(\epsilon)A_0(\epsilon))=2$ and $\lambda_{\min}^+(A^\top(\epsilon)A(\epsilon))=2-\sqrt{\epsilon^4+4}+\epsilon^2=O(\epsilon^2)$.
\end{eg} 
\begin{eg}\label{ex:eg3-8}
	Consider a tridiagonal matrix
	\[
	A(\epsilon)=\begin{bmatrix}
		1+\epsilon&-1&0\\
		-1&2+\epsilon&-1\\
		0&-1&1+\epsilon\\
	\end{bmatrix},
	\]
	with the splitting 
	\[
	A_0(\epsilon) = \begin{bmatrix}
		1+\epsilon&-1&0\\
		-1&2+\epsilon&-1	 	
	\end{bmatrix},\quad A_1(\epsilon) =[0, -1, 1+\epsilon].
	\]
	It can be verified that $A(\epsilon)$ is nearly singular and  
	%%%%%%%%%%%%%%%%%%%A0A'%%%%%%%%%%%%%%%%%%%
	%	 \[
	%    A_0(\epsilon)A^\top(\epsilon)  = \begin{bmatrix}
		%	 	\epsilon^2+2\epsilon+2&-3-2\epsilon&1\\
		%	 	-3-2\epsilon&\epsilon^2+4\epsilon+6&-3-2\epsilon
		%	 \end{bmatrix} .
	% \] 
	%%%%%%%%%%%%%%%%%%%A0A'%%%%%%%%%%%%%%%%%%% 
	%%%%%%%%%%%%%%%%%%%%%%%%%%%%%%kerAe%%%%%%%%%%%%%%%%%
	%\begin{equation}\label{key}
	%		 \footnotesize
	%	\appker{A(\epsilon)}={}\myspan \left\{\left[ \frac{s \left(s^2+4 s+3\right)}{s^4+6 s^3+12s^2+8 s+3} , \frac{s {\left(s+1\right)}^2 \left(s+3\right)}{s^4+6 s^3+12 s^2+8 s+3} , \frac{s \left(s^4+7 s^3+16 s^2+13 s+3\right)}{s^4+6 s^3+12 s^2+8 s+3}\right]^\top\right\}.
	%\end{equation}
	%%%%%%%%%%%%%%%%%%%%%%%%%%%%%%kerAe%%%%%%%%%%%%%%%%%
	\[
	\begin{aligned}
		\appker{A(\epsilon)}={}&\myspan \left\{\left[1-4/3\epsilon+ O(\epsilon^2),1-\epsilon/3+O(\epsilon^2)  ,1+5/3\epsilon+O(\epsilon^2) \right]^\top\right\},\\ 
		%%%%%%%%%%%%%%%%%%%%%%%dual-kerAe%%%%%%%%%%%%%%%%%%%%%%%%
		%	\dappker{A(\epsilon)}={}&\myspan \left\{\left[ \frac{3s^2+8s+3}{s^4+6s^3+12s^2+8s+3}, \frac{2s^3+7s^2+8s+3}{s^4+6s^3+12s^2+8s+3} , 1 
		%	\right]^\top\right\},\\
		%%%%%%%%%%%%%%%%%%%%%%%dual-kerAe%%%%%%%%%%%%%%%%%%%%%%%%
		\dappker{A(\epsilon)}={}&\myspan \left\{\left[1+3\epsilon^2+O(\epsilon^3),1+5/3\epsilon^2+O(\epsilon^3),1\right]^\top\right\}.
	\end{aligned}
	\]
	Moreover, we have $\lambda_{\min}^+(A^\top_0(\epsilon)A_0(\epsilon))=4-\sqrt{13}+O(\epsilon)$ and $\lambda_{\min}^+(A^\top(\epsilon)A(\epsilon))= \epsilon^2$.
	%%%%%%%%%%%%%%%%%%%%%computation%%%%%%%%%%%%%%%%%%%%%%
	%	 \[ W_\epsilon={\rm span}\left\{\begin{bmatrix}
		%	 	\frac{3\varepsilon^2 + 8\varepsilon + 3}{\varepsilon^4 + 6\varepsilon^3 + 12\varepsilon^2 + 8\varepsilon + 3}, &
		%	 	\frac{(2\varepsilon + 3)(\varepsilon^2 + 2\varepsilon + 1)}{\varepsilon^4 + 6\varepsilon^3 + 12\varepsilon^2 + 8\varepsilon + 3} ,&
		%	 	1
		%	 \end{bmatrix}^\top\right\},\]
	%	 \[\bar W_\epsilon={\rm span}\left\{\begin{bmatrix}
		%	 	\frac{ (\epsilon + 1) (3\epsilon^{2} + 8\epsilon + 3) - (2\epsilon + 3) (\epsilon^{2} + 2\epsilon + 1) }{ \epsilon^{4} + 6\epsilon^{3} + 12\epsilon^{2} + 8\epsilon + 3 }\\
		%	 	\frac{ (2\epsilon + 3) (\epsilon + 2) (\epsilon^{2} + 2\epsilon + 1) - (3\epsilon^{2} + 8\epsilon + 3) }{ \epsilon^{4} + 6\epsilon^{3} + 12\epsilon^{2} + 8\epsilon + 3 } - 1 \\
		%	 	\epsilon - \frac{ (2\epsilon + 3) (\epsilon^{2} + 2\epsilon + 1) }{ \epsilon^{4} + 6\epsilon^{3} + 12\epsilon^{2} + 8\epsilon + 3 } + 1
		%	 \end{bmatrix}\right\}.\]
	%%%%%%%%%%%%%%%%%%%%%computation%%%%%%%%%%%%%%%%%%%%%%
\end{eg} 
\begin{eg}  Consider  
	\[A(\epsilon)=\begin{bmatrix}
		1&-1\\1+\epsilon&-1+\epsilon\\2&-2
	\end{bmatrix},\quad\text{with}\,\,
	A_0(\epsilon)=[1,-1]
	,\,
	A_1(\epsilon) =\begin{bmatrix}
		1+\epsilon&-1+\epsilon\\
		2 & -2
	\end{bmatrix}.
	\]
	It is clear that $A(\epsilon)$ is nearly singular. A direct computation leads to
	\[
	\dappker{A(\epsilon)} = \myspan\left\{
	[-1,1,0]^\top,
	[-2,0,1]^\top
	\right\}
	,\quad \appker{A(\epsilon)} = \myspan\left\{[1,1]^\top\right\}.
	\]
	We have $\lambda_{\min}^+(A^\top_0(\epsilon)A_0(\epsilon))=2$ and $\lambda_{\min}^+(A^\top(\epsilon)A(\epsilon))= O(\epsilon^2)$. Note that $A(\epsilon)$ is not full row rank and we have \( \dim \appker{A(\epsilon)} < \dim 		\dappker{A(\epsilon)}\).
\end{eg}

\section{A Robust Kaczmarz Method}
\label{sec:ssc-kacz}
From now on, we focus on the nearly singular linear system \cref{eq:Ax=b}, where $b\in\ran{A(\epsilon)}$ and  $A:[0,\bar{\epsilon}]\to\R^{m\times n}$ satisfies \cref{assum:A0-A1,def:nsp}.
\subsection{The classical Kaczmarz}
The classical Kaczmarz for solving \cref{eq:Ax=b} is
\begin{equation} 
	\label{eq:kacz}
	%	\tag{Kaczmarz}
	%	\left\{
	\begin{aligned}
		%		{}&	v= x_k,\\
		{}&	x_{k+1} = x_k + \omega\frac{b_{i} - A_{(i)}(\epsilon) x_k}{\| A_{(i)}(\epsilon)\|^{2}} A_{(i)}^\top(\epsilon),
		%		\quad 1\leq i\leq m,\omega>0\\
		%		{}&	x_{k+1} =v.
	\end{aligned}
	%	\right.
\end{equation}
where $\omega>0$ denotes the relaxation parameter; see \cite{moorman2021randomized}. Recall that \cref{eq:Ax=b} is consistent since $b\in \ran{A(\epsilon)}$. In this case,  the general solution is given by $x^*=x_{LS}+\widehat{x}$, where $x_{LS}=A^+b$ and $\widehat{x}\in \myker{A(\epsilon)}$ is arbitrary. It is known that \cite{oswald_convergence_2015} the classical Kaczmarz method converges to $x_{LS}$ provided that $x_0\in \row{A(\epsilon)}$. Following \cite[Example 2]{oswald_convergence_2015}, we can reformulate \cref{eq:kacz} into the SSC framework (cf.\cref{sec:ssc}) with the following setting:
%	\begin{itemize}
	%		\item $V:= \row{A(\epsilon)}$ with the Euclidean inner product $(x,y) :=x^\top y $;
	%		\item $A:=I$ is the identity operator and $b:= x_{LS}$;
	%		\item $V= \sum_{i=1}^{m}V_i$ with $V_i:=\R$ and $R_i:=A_{(i)}^\top(\epsilon):V_i\to V$ for $1\leq i\leq m$;
	%		\item $a_i(x_i,y_i)=d_ix_iy_i$ with $d_i=\nm{A_{(i)}(\epsilon)}^2/\omega$ for $1\leq i\leq m$.
	%	\end{itemize}	
\begin{equation}\label{eq:set-kacz-ssc}
	\renewcommand{\arraystretch}{1.2}
	\begin{array}{l}
		\textit{{\footnotesize$\bullet$}\,\, $V:= \row{A(\epsilon)}$ with the Euclidean inner product $(x,y) :=x^\top y $;}\\
		\textit{{\footnotesize$\bullet$}\,\, $A:=I:V\to V$ is the identity operator and $f:= x_{LS}$;}\\
		\text{{\footnotesize$\bullet$}\,\, $V= \sum_{i=1}^{m}V_i$ with $V_i:=\R$ and $R_i:=A_{(i)}^\top(\epsilon):V_i\to V$ for $1\leq i\leq m$;}\\
		\textit{{\footnotesize$\bullet$}\,\, $a_i(x_i,y_i)=d_ix_iy_i$ with $d_i=\nm{A_{(i)}(\epsilon)}^2>0$ for $1\leq i\leq m$.}
	\end{array}
\end{equation}

With this, it is easy to get $T_i = R_i^* =   A_{(i)}(\epsilon)/d_i,\,\mathcal R = A^\top(\epsilon)$ and $ \mathcal R^T=\mathcal R^\star =  D_\epsilon^{-1}A(\epsilon)$
%\[
% \mathcal R^T=\mathcal R^\star =  D_\epsilon^{-1}A(\epsilon),\,\bm D = \diag{D_\epsilon^{-1}A(\epsilon)A^\top(\epsilon)},\, \bm U = \triu{D_\epsilon^{-1}A(\epsilon)A^\top(\epsilon)},
%\]
with $D_\epsilon = \diag{d_1,\cdots,d_m} = \diag{ A(\epsilon)A^\top(\epsilon)}\succeq \eta_0I$ (cf.\cref{eq:d0-d1}). We have the stationary iteration form
\begin{equation}
	\label{eq:iter-Bssc-kacz}
	x_{k+1} = x_k + B_{\rm kacz}(\omega)(x_{ LS}-x_k),
\end{equation}
where $B_{\rm kacz}(\omega):V\to V$ satisfies $	I-B_{\rm kacz}(\omega) =(I-\omega R_mT_m)\cdots(I-\omega R_1T_1)$.
Since $
d_i>0$ (cf. \cref{assum:A0-A1}),  Assumptions \ref{assum:V-Vj}, and \ref{assum:aj-Vj} hold true.

Under the setting \cref{eq:set-kacz-ssc}, the product space $\bm V = \R^m$ has the inner product $\inner{\bm u,\bm v}_{\bm V}:=\inner{\bm u,\bm v}_{D_\epsilon}$.  Define 
\[\small
\begin{aligned}
	\delta_{\max}(A(\epsilon)):={}&\sup_{\substack{v\in\row{A(\epsilon)}\\\nm{v}=1}}(\mathcal R^T v,\mathcal R^Tv)_{\bm V},\quad 
	\delta_{\min}(A(\epsilon)):={} \inf_{\substack{v\in\row{A(\epsilon)}\\\nm{v}=1}}(\mathcal R^T v,\mathcal R^Tv)_{\bm V}.
\end{aligned}
\]
Then we have (cf.\cite[Example 2]{oswald_convergence_2015})
\begin{equation}\label{eq:sigma-max-min}
	\delta_{\min}(A(\epsilon))={}\lambda^+_{\min}(A^\top(\epsilon)D_\epsilon^{-1}A(\epsilon)),\quad \delta_{\max}(A(\epsilon))={}\lambda_{\max}(A^\top(\epsilon)D_\epsilon^{-1}A(\epsilon)).
\end{equation}  According to \cite[Theorem 1]{oswald_convergence_2015}, if $\omega\in(0,2/	\delta_{\max}(A(\epsilon) )$, then \cref{assum:TjRj} holds true and we have 
\begin{equation}\label{eq:conv-kacz}
	\nm{I-B_{\rm kacz}(\omega)}^2_{\row{A(\epsilon)}\to \row{A(\epsilon)}} \leq 1-\frac{4(2-\omega\delta_{\max}(A(\epsilon) ))	\delta_{\min}(A(\epsilon))}{\left(  \omega\delta_{\max}(A(\epsilon) )  \lfloor \log_{2}(2m) \rfloor  + 2 \right)^{2}}.
	%\frac{	\delta_{\min}(A(\epsilon))}{1+\lfloor \log_{2}(2m) \rfloor}.
\end{equation}
This together with the  contraction estimate (cf.\cref{eq:iter-Bssc-kacz})
\[
\nm{x_{k+1}-x_{LS}}\leq	\nm{I-B_{\rm kacz}(\omega)}_{\row{A(\epsilon)}\to \row{A(\epsilon)}}\nm{x_k-x_{LS}}
\]
yields the linear rate 
\[
\nm{x_k-x_{LS}}^2\leq \left(1-\frac{4(2-\omega\delta_{\max}(A(\epsilon) ))	\delta_{\min}(A(\epsilon))}{\left(  \omega\delta_{\max}(A(\epsilon) )  \lfloor \log_{2}(2m) \rfloor  + 2 \right)^{2}}\right)^k\nm{x_0-x_{LS}}^2.
\]
An optimal choice of the relaxation parameter \cite[Eq.(27)]{oswald_convergence_2015} 
\[
\omega^* =\frac{2}{\delta_{\max}(A(\epsilon))(2+\lfloor \log_{2}(2m) \rfloor)}
\]
%$\omega^* =2/\delta_{\max}(A(\epsilon))/(2+\lfloor \log_{2}(2m) \rfloor)$
%\[
%\omega^* =\frac{2}{2+\lfloor \log_{2}(2m) \rfloor},
%%\quad\Longrightarrow\quad \nm{I-B_{\rm kacz}(\omega^*)}^2_{V\to V} \leq1-
%%\frac{	\delta_{\min}(A(\epsilon))}{1+\lfloor \log_{2}(2m) \rfloor},
%\]
gives
\begin{equation}\label{eq:rate-kacz}
	\nm{x_k-x_{LS}}^2\leq \left( 1-
	\frac{	\delta_{\min}(A(\epsilon))}{\delta_{\max}(A(\epsilon))(1+\lfloor \log_{2}(2m) \rfloor)}\right)^k\nm{x_0-x_{LS}}^2.
\end{equation}
However,  it is clear 
%that Invoking the definition \cref{eq:RT} of $\mathcal R^T$ and our setting \cref{eq:set-kacz-ssc}, it is clear 
that $\delta_{\min}(A(\epsilon))={}\lambda^+_{\min}(A^\top(\epsilon)D_\epsilon^{-1}A(\epsilon))\leq 1/\eta_0\lambda^+_{\min}(A^\top(\epsilon)A(\epsilon))$. Thus, by \cref{lem:up-bd-lambda-min+}, for nearly singular systems, $\lambda^+_{\min}(A^\top(\epsilon)A(\epsilon))\to0$ as $\epsilon\to 0+$ and the rate of the Kaczmarz iteration \cref{eq:kacz} deteriorates. This can be can observed from \cref{table-comparison}, and we refer to \cref{sec:experiments} for more numerical evidences.
\subsection{From SSC to robust Kaczmarz}
\label{sec:kak}
%In view of \cref{eq:conv-kacz}, the convergence rate of the classical Kaczmarz iteration \cref{eq:kacz} depends on the smallest nonzero eigenvalue of $A^\top(\epsilon)D_\epsilon^{-1}A(\epsilon)$, which, however, approaches to zero as $\epsilon\to0$ because of the nearly singular property of $A(\epsilon)$. Therefore, if $\epsilon$ is decreasing, then the number of iterations for satisfying a given tolerance grows dramatically. This can be can observed from \cref{tab:exp-ns,fig:combined2}.

To overcome the near singularity, motivated by \cite{lee_robust_2007}, we add the
%For the case where the matrix $A$ in \cref{eq:Ax=b} is nearly singular and full of row rank, we add the 
approximate kernel $\appker{A(\epsilon)}$ (cf.\cref{def:ak}) to the space decomposition and consider the  SSC setting:
\begin{equation}\label{eq:set-kak-ssc}\small
	\renewcommand{\arraystretch}{1.2}
	\begin{array}{l}
		\textit{{\footnotesize$\bullet$}\,\, $V:= \row{A(\epsilon)}$ with the Euclidean inner product $(x,y) :=x^\top y $;}\\
		\textit{{\footnotesize$\bullet$}\,\, $A:=I:V\to V$ is the identity operator and $f:= x_{LS}$;}\\
		\text{{\footnotesize$\bullet$}\,\, $V= \sum_{i=1}^{m+1}V_i$ with $V_i:=\R$ for $1\leq i\leq m$ and $V_{m+1}:=\appker{A(\epsilon)}$;}\\
		\text{{\footnotesize$\bullet$}\,\,  $R_i:=A_{(i)}^\top(\epsilon):V_i\to V$ for $1\leq i\leq m$ and $R_{m+1} = \iota:V_{m+1}\to V$;}\\		
		\textit{{\footnotesize$\bullet$}\,\, $a_i(x_i,y_i)=d_ix_iy_i$ with $d_i=\nm{A_{(i)}(\epsilon)}^2$ for $1\leq i\leq m$ and $a_{m+1}(x,y) = x^\top y$.}
	\end{array}
\end{equation}

%$R_{m+1}^*=w_j(w_j^\top w_j)^{-1}w_j^\top$ 
With this preparation, it is easy to check Assumptions \ref{assum:V-Vj} and \ref{assum:aj-Vj}, and we find that $T_i = R_i^* =   A_{(i)}(\epsilon)/d_i$ for $1\leq i\leq m$ and $T_{m+1}=R_{m+1}^*:=R:\row{A(\epsilon)}\to \appker{A(\epsilon)}$ is the orthogonal projection. In this case,  we have $\frac2\omega I-T_iR_i = \frac2\omega-1>0$ for all $1\leq i\leq m+1$, and  \cref{assum:TjRj} holds true. This leads to an SSC presentation: given $x_k\in \row{A(\epsilon)}$, set $v_{k,1} = x_k$ and update $x_{k+1}=v_{k,m+2}$ by 
\begin{equation}\label{eq:kak}
	%	\small
	\tag{KaK}
	\left\{
	\begin{aligned}
		{}&		v_{k,i+1} 
		%	= v_{k,i} +\omega R_iR_i^*(x_{LS}-v_{k,i})
		=v_{k,i} + \omega\frac{b_{i} - A_{(i)}(\epsilon) v_{k,i}}{\| A_{(i)}(\epsilon)\|^{2}} A_{(i)}^\top(\epsilon),\quad\,1\leq i\leq m,\\
		{}&		v_{k,m+2}
		%	=v_{k,m+1}+\omega R_{m+1}R^*_{m+1}(x_{LS}-v_{k,m+1})
		=v_{k,m+1}+\omega R (x_{LS}-v_{k,m+1}),
	\end{aligned}
	\right.
\end{equation}
where $\omega\in(0,2)$. As an extension of Kaczmarz \cref{eq:kacz},  we called it the {\it Kernel-augmented Kaczmarz} (KaK) method, which is equivalent to a stationary iteration
\begin{equation}\label{eq:iter-Bssc-robu-kacz}
	x_{k+1} = x_k + B_{\rm kak}(\omega)(x_{LS}-x_k),
\end{equation}
where  the iterator $B_{\rm kak}(\omega):\row{A(\epsilon)}\to \row{A(\epsilon)}$ satisfies 
\[
I-B_{\rm kak}(\omega) =(I-\omega R_{m+1}R_{m+1}^*)(I-B_{\rm kacz}(\omega))
\]
%\begin{assum}\label{assum:d0-d1}
%	Assume that there exist two positive constants $d_0,d_1>0$ in dependent on $\epsilon$ such that $\min_{1\leq i\leq m}\nm{A_{(i)}(\epsilon)}\geq d_0$  and $ \lambda_{\max}(A^\top(\epsilon)D_\epsilon^{-1}A(\epsilon))\leq d_1$ for all $\epsilon\in[0,\infty)$.
%\end{assum}

Based on \cref{thm:xz-id}, we have the following estimate, which paves the way for proving the uniform rate of convergence of \cref{eq:kak}.
\begin{lem}\label{lem:app-xz-kak}
	With the setting \cref{eq:set-kak-ssc} and $\omega\in(0,2/(1+\delta_{\max}(A(\epsilon))))$, for the kernel-augmented Kaczmarz method \cref{eq:kak}, we have
	\begin{equation}\label{eq:app-xz-kak}
		\nm{I-B_{\rm kak}(\omega)}^2_{\row{A(\epsilon)}\to \row{A(\epsilon)}} \leq 1-C_0(\epsilon,\omega)\inf_{v\in\row{A(\epsilon)}\backslash\{0\}}\frac{(\mathcal R\mathcal R^T v,v)}{(v,v)},
	\end{equation}
	where
	\begin{equation}\label{eq:C0-e-w}
		C_0(\epsilon,\omega) :={}	\frac{4\omega(2-\omega(1+\delta_{\max}(A(\epsilon))))}{\left(2+  \omega  \lfloor \log_{2}(2m) \rfloor \delta_{\max}(A(\epsilon)) + 2\omega\nm{D_\epsilon^{-1}A(\epsilon)} \right)^{2}}.
	\end{equation}
	The optimal choice 
	\begin{equation}\label{eq:opt-w}
		\omega^*  =\frac{2}{2+(2+\lfloor\log _{2}(2m)\rfloor)\delta_{\max}(A(\epsilon))+2\nm{D_\epsilon^{-1}A(\epsilon)}}
	\end{equation}
	leads to the maximal value
	\[
	C_0(\epsilon,\omega^*)=\frac{1}{1+(1+\lfloor\log _{2}(2m)\rfloor)\delta_{\max}(A(\epsilon))+2\nm{D_\epsilon^{-1}A(\epsilon)}}.
	%	1-\frac{\inf_{v\in\row{A(\epsilon)}\backslash\{0\}}\frac{(\mathcal R\mathcal R^T v,v)}{(v,v)}}{2+\left\lfloor\log _{2}(2 m)\right\rfloor+2\nm{D_\epsilon^{-1}A(\epsilon)}}.
	\]
\end{lem}
\begin{proof}
	It is clear that 
	\begin{equation}\label{eq:bmA}\small
		\mathcal R = [A^\top(\epsilon),I],\,\, \mathcal R^T=\mathcal R^\star =  \begin{bmatrix}
			D_\epsilon^{-1}A(\epsilon)\\ R_{m+1}R_{m+1}^*
		\end{bmatrix},\,
		\bm A = \begin{bmatrix}
			D_\epsilon^{-1}A(\epsilon)A^\top(\epsilon)&D_\epsilon^{-1}A(\epsilon)\\R_{m+1}R_{m+1}^*A^\top(\epsilon) &R_{m+1}R_{m+1}^*
		\end{bmatrix}.
	\end{equation}
	Let $\bm D_0 :=	\diag{D_\epsilon^{-1}A(\epsilon)A^\top(\epsilon)}$ and $\bm U_0:=	\triu{D_\epsilon^{-1}A(\epsilon)A^\top(\epsilon)}$, then we have the block decomposition
	\[
	\bm D = \begin{bmatrix}
		\bm D_0&O\\O&R_{m+1}R_{m+1}^*
	\end{bmatrix},\quad 
	\bm U = \begin{bmatrix}
		\bm U_0&D_\epsilon^{-1}A(\epsilon)\\
		O&O
	\end{bmatrix}. 
	\]
	Let $\bm V_0=\R^m$. The product space $\bm V = \bm V_0\times V_{m+1}$ has the inner product $\inner{\bm u,\bm v}_{\bm V}:=\inner{\bm u_0,\bm v_0}_{D_\epsilon}+\inner{u,v}$, for any $\bm u=(\bm u_0,u)\in \bm V$ and $\bm v=(\bm v_0,v)\in \bm V$. Recall that $R_{m+1}=\iota:V_{m+1}\to V$ and $T_{m+1}R_{m+1}v=T_{m+1}v$ for all $v\in V_{m+1}$. 
	
	Thanks to \cite[Eq.(23)]{oswald_convergence_2015}, we have 
	$\nm{\bm D_0}_{\bm V_0\to\bm V_0}\leq \delta_{\max}(A(\epsilon))$. Therefore, we have 
	\[
	\begin{aligned}
		\inner{\bm D\bm v,\bm D\bm v}_{\bm V} ={}&	\inner{\bm D_0\bm v_0,\bm D_0\bm v_0}_{\bm V_0}+a_{m+1}(T_{m+1}v,T_{m+1}v)\\
		\leq{}&\delta^2_{\max}(A(\epsilon))\nm{\bm v_0}^2_{\bm V_0}+\nm{ T_{m+1}v}^2\\ 
		\leq{}&\left(1+\delta^2_{\max}(A(\epsilon))\right)\left(\nm{\bm v_0}^2_{\bm V_0}+\nm{ R^*_{m+1}v}^2\right)=\left(1+\delta^2_{\max}(A(\epsilon))\right)\nm{\bm v}_{\bm V}^2,
		%	\\ 	 ={}&a_{m+1}(T_{m+1}v,T_{m+1}v)+\sum_{i=1}^{m}a_i(T_iR_i\bm v_0^i,T_iR_i\bm v_0^i) \\
		%	={}&\nm{T_{m+1}v}^2+\sum_{i=1}^{m} \inner{R_iT_iR_i\bm v_0^i,R_i\bm v_0^i}\\
		%	%	\leq {}&\sum_{i=1}^{m} \inner{\left(\sum_{j=1}^{m}R_jT_j\right)R_i\bm v_i,R_i\bm v_i} = \sum_{i=1}^{m}\inner{\mathcal R\mathcal R^T R_i\bm v_i,R_i\bm v_i}\\
		%	%=	{}&\sum_{i=1}^{m}\inner{\mathcal R^T R_i\bm v_i,\mathcal R^T R_i\bm v_i}_{\bm V}\leq \delta_{\max}(A(\epsilon))\sum_{i=1}^{m}\inner{R_i\bm v_i, R_i\bm v_i}\\
		%	={}&\inner{T_{m+1}v,v}+\sum_{i=1}^{m} \inner{R_i\bm v_0^i,R_i\bm v_0^i}\quad\left(T_iR_i = 1,\,1\leq i\leq m\right)\\
		%	={}&\sum_{i=1}^{m+1}a_i\inner{T_iR_i\bm v_i,  \bm v_i}=\inner{\bm D\bm v,\bm v}_{\bm V},
	\end{aligned}
	\]
	which implies $\nm{\bm D}_{\bm V\to\bm V}\leq 1+\delta_{\max}(A(\epsilon))$.  
	Since $\omega\in(0,2/(1+\delta_{\max}(A(\epsilon))))$, $2\bm I-\omega\bm D$ is SPD and in view of \cref{thm:xz-id}, we have
	$	\nm{I-B_{\rm kak}(\omega)}^2_{\row{A(\epsilon)}\to \row{A(\epsilon)}} = 1-\frac{1}{c_0(\omega)}$, where
	\begin{equation}\label{eq:c0omega}\small
		\begin{aligned}
			c_0(\omega) ={}&\frac{1}{\omega} \sup_{v\in \row{A(\epsilon)}\backslash\{0\}} \inf_{\mathcal{R}\bm v =v}\frac{\inner{	\left(2\bm I-\omega\bm D\right)^{-1}	\left(\bm I+\omega\bm U\right)\bm v,	\left(\bm I+\omega\bm U\right)\bm v}_{\bm V}}{\inner{v,v}}\\
			\leq{}	&\frac{1}{\omega(2-\omega\nm{\bm D}_{\bm V\to\bm V})} \sup_{v\in \row{A(\epsilon)}\backslash\{0\}} \inf_{\mathcal{R} \bm v =v}\frac{\inner{ 	\left(\bm I+\omega \bm U\right)\bm v,	\left(\bm I+\omega \bm U\right)\bm v}_{\bm V}}{\inner{v,v}}.
		\end{aligned}
	\end{equation}
	In addition, noticing $\bm U'_0=\bm L'_0$ and the fact $\nm{\bm L_0}_{\bm V_0\to\bm V_0}\leq\lfloor \log_{2}(2m) \rfloor/2 \delta_{\max}(A(\epsilon)) $ (cf.\cite[Eq.(25)]{oswald_convergence_2015}), it follows that 
	\[
	\begin{aligned}
		\nm{\bm U}_{\bm V\to\bm V}^2 ={}& \sup_{\substack{\bm v=(\bm v_0,v)\in \bm V\\\nm{\bm v}_{\bm V}=1}}\nm{\bm U\bm v}_{\bm V}^2=\sup_{\substack{\bm v=(\bm v_0,v)\in \bm V\\\nm{\bm v}_{\bm V}=1}}\left(\nm{\bm U_0\bm v_0}_{\bm V_0}^2+\nm{D_\epsilon^{-1}A(\epsilon)v}^2\right)\\
		\leq{}&\sup_{\substack{\bm v=(\bm v_0,v)\in \bm V\\\nm{\bm v}_{\bm V}=1}}\left( \frac{\lfloor \log_{2}(2m) \rfloor^2}{4}  \delta^2_{\max}(A(\epsilon))\nm{\bm v_0}_{\bm V_0}^2+\nm{D_\epsilon^{-1}A(\epsilon)}^2\nm{v}^2\right)\\
		\leq{}&\frac{\lfloor \log_{2}(2m) \rfloor^2}{4}  \delta^2_{\max}(A(\epsilon))+\nm{D_\epsilon^{-1}A(\epsilon)}^2,
	\end{aligned}
	\]
	%Since by \cref{eq:d0-d1} we have 
	%\[
	%\nm{D_\epsilon^{-1}A(\epsilon)}^2\leq \nm{D_\epsilon^{-1}}\nm{D_\epsilon^{-1/2}A(\epsilon)}^2\leq 
	%%=\nm{D_\epsilon^{-1}}\lambda_{\max}(A^\top(\epsilon)D_\epsilon^{-1}A(\epsilon))\leq 
	%\lambda_{\max}(A^\top(\epsilon)D_\epsilon^{-1}A(\epsilon))/\eta_0,
	%\]
	which gives $	\nm{\bm U}_{\bm V\to\bm V}\leq \frac{1}{2} \lfloor \log_{2}(2m) \rfloor\delta_{\max}(A(\epsilon))+\nm{D_\epsilon^{-1}A(\epsilon)}$.
	%\begin{equation*}
	%%	\label{est-U}
	%	\nm{\bm U}_{\bm V\to\bm V}\leq \frac{1}{2} \lfloor \log_{2}(2m) \rfloor+\nm{D_\epsilon^{-1}A(\epsilon)}.
	%\end{equation*}
	Hence, plugging the above estimates into \cref{eq:c0omega}, we get
	\begin{align*}	c_0(\omega)  
		%	\leq{}	&\frac{1}{\omega(2-\omega\nm{\bm D}_{\bm V\to\bm V})} \sup_{v\in \row{A(\epsilon)}\backslash\{0\}} \inf_{\mathcal{R} \bm v =v}\frac{\inner{ 	\left(\bm I+\omega \bm U\right)\bm v,	\left(\bm I+\omega \bm U\right)\bm v}_{\bm V}}{\inner{v,v}}\\
		\leq{}	&
		%	\frac{\left(  \omega  \lfloor \log_{2}(2m) \rfloor  + 2\omega\sqrt{d_1/d_0}+2 \right)^{2}}{4\omega(2-\omega)} 
		\frac{\left(1+\omega\nm{\bm U}_{\bm V\to\bm V}\right)^2}{\omega(2-\omega\nm{\bm D}_{\bm V\to\bm V})}
		\sup_{v\in \row{A(\epsilon)}\backslash\{0\}} \inf_{\mathcal{R} \bm v =v}\frac{\inner{\bm v,\bm v}_{\bm V}}{\inner{
				v,v}}\\
		\leq{}	&
		%	\frac{\left(  \omega  \lfloor \log_{2}(2m) \rfloor  + 2\omega\sqrt{d_1/d_0}+2 \right)^{2}}{4\omega(2-\omega)} 
		\frac{1}{C_0(\epsilon,\omega)}
		\sup_{v\in \row{A(\epsilon)}\backslash\{0\}} \inf_{\mathcal{R} \bm v =v}\frac{\inner{\bm v,\bm v}_{\bm V}}{\inner{
				v,v}}.
		%	={}	&\frac{\left(  \omega  \lfloor \log_{2}(2m) \rfloor  + 2\omega\sqrt{d_1/d_0}+2 \right)^{2}}{4\omega(2-\omega)} \left(\inf_{v\in\row{A(\epsilon)}\backslash\{0\}}\frac{(\mathcal R\mathcal R^T v,v)}{(v,v)}\right)^{-1}.
	\end{align*}
	Since $\mathcal{R}:\bm V\to V$ is surjective and $\mathcal R^\star=\mathcal R^T$, by \cref{lem:inv-batB}, we see that
	\[
	\begin{aligned}
		{}&	\sup_{v\in \row{A(\epsilon)}\backslash\{0\}}\inf_{\mathcal{R}\bm v = v}\frac{(\bm v,\bm v)_{\bm V}}{(v,v)}=\sup_{v\in \row{A(\epsilon)}\backslash\{0\}}\frac{\left( (\mathcal{R}\mathcal{R}^{T})^{-1}v, v \right)}{\inner{v,v}}\\
		={}&\left(\inf_{v\in \row{A(\epsilon)}\backslash\{0\}}\frac{(\mathcal{R}\mathcal{R}^{T}v,v)}{\inner{v,v}}\right)^{-1}
		={}\left(\inf_{v\in \row{A(\epsilon)}\backslash\{0\}}\frac{(\mathcal{R}^{T}v,\mathcal{R}v)_{\bm V}}{\inner{v,v}}\right)^{-1}.
	\end{aligned}
	\]
	This together with the previous inequality gives \cref{eq:app-xz-kak} and concludes the proof.
\end{proof}

The final convergence result of \cref{eq:kak} is given by the following theorem, which implies a uniform contraction number for small $\epsilon$; see \cref{rem:rho0} for more discussions.
\begin{thm}
	\label{thm:conv_robust}
	Let  $\omega\in(0,2/(1+\delta_{\max}(A(\epsilon))))$ and define \begin{equation}\label{eq:rho}
		\rho(\epsilon,\omega):=C_0(\epsilon,\omega)\left( 	\min\left(1,\sigma_0(\epsilon)/\eta_1\right)-C_1(\epsilon)/\eta_0\right),
	\end{equation}
	where $\sigma_0(\epsilon)$  and $C_0(\epsilon,\omega)$ are defined respectively by \cref{eq:sigma0,eq:C0-e-w} and
	\[
	C_1(\epsilon) :=  \nm{E_1}^2+2\nm{A_1(0)}\nm{E_1}+\nm{E_0}\nm{A_1(0)}\nm{\Pi},
	\]
	with $E_0:=A_0(\epsilon)-A_0(0)$ and $E_1:=A_1(\epsilon)-A_1(0)$. 
	For \cref{eq:kak}, we have
	\begin{equation}\label{eq:r-est-1}
		\nm{I-B_{\rm kak}(\omega)}^2_{\row{A(\epsilon)}\to \row{A(\epsilon)}} \leq  1-\rho(\epsilon,\omega).
	\end{equation}
	Moreover, the optimal choice \cref{eq:opt-w} gives 
	\[
	%		\nm{I-B_{\rm kak}(\omega^*)}^2 \leq 1-\rho(\epsilon,\omega^*),\quad 
	\rho(\epsilon,\omega^*)= \frac{\min\left(1,\sigma_0(\epsilon)/\eta_1\right)-C_1(\epsilon)/\eta_0}{1+(1+\lfloor\log _{2}(2m)\rfloor)\delta_{\max}(A(\epsilon))+2\nm{D_\epsilon^{-1}A(\epsilon)}}.
	\]
\end{thm}
\begin{proof}
	According to \cref{eq:app-xz-kak}, it is sufficient to prove
	\begin{equation}\label{eq:C1}
		\inf_{v\in\row{A(\epsilon)}\backslash\{0\}}\frac{(\mathcal R\mathcal R^T v,v)}{(v,v)}\geq \min\left(1,\sigma_0(\epsilon)/\eta_1\right)-C_1(\epsilon)/\eta_0.
	\end{equation}
	From \cref{lem:lambda-min-A0}, the orthogonal complement of $\row{A_0(\epsilon)}$ in $\row{A(\epsilon)}$ is the approximate kernel $\appker{A(\epsilon)}$.
	Thus for any \( v \in \row{A(\epsilon)} \) we have the orthogonal decomposition \( v = v_n + v_r \), where \( v_n \in \appker{A(\epsilon)}\), \( v_r \in \row{A_0(\epsilon)}\) and \( \inner{v_r, v_n } = 0 \). With this, we find that $\nm{v}^2=\nm{v_r}^2+\nm{v_n}^2$ and 
	\begin{align*}\left( \mathcal{R}\mathcal{R}^{T}v, v \right)
		={}	& \left( \mathcal{R}\mathcal{R}^{T}v_r, v_r\right)+2\left(\mathcal{R} \mathcal{R}^{T}v_r, v_n\right)+\left( \mathcal{R}\mathcal{R}^{T}v_n, v_n\right)\\
		={}	& \left( \mathcal{R}^{T}v_r, \mathcal{R}^Tv_r\right)_{\bm V}+2\left(\mathcal{R} \mathcal{R}^{T}v_r, v_n\right)+\left( \mathcal{R}^{T}v_n, \mathcal R^Tv_n\right)_{\bm V}\\
		= {}	& \nm{\mathcal{R}^{T}v_r}^2_{\bm V}+2\left(\mathcal{R} \mathcal{R}^{T}v_r, v_n\right)+\nm{\mathcal{R}^{T}v_n}^2_{\bm V}.
	\end{align*} 
	It follows from 	 \cref{eq:d0-d1,lem:lambda-min-A0,eq:bmA} that
	\begin{equation}\label{eq:lb-vr}
		\begin{aligned}
			\nm{\mathcal{R}^{T}v_r}^2_{\bm V}={}&\nm{D_\epsilon^{-1}A(\epsilon)v_r}_{D_\epsilon}^2+\nm{R_{m+1}R_{m+1}^*v_r}^2\\
			={}&v_r^\top A^\top(\epsilon)D_\epsilon^{-1}A(\epsilon)v_r
			\geq{}\frac{1}{\eta_1}v_r^\top A^\top(\epsilon)A(\epsilon)v_r\\
			={}&\frac{1}{\eta_1}v_r^\top \left(A_0^\top(\epsilon)A_0(\epsilon)+A_1^\top(\epsilon)A_1(\epsilon)\right)v_r\geq\frac{\sigma_0(\epsilon)}{\eta_1}\nm{v_r}^2,
		\end{aligned}
	\end{equation}
	where we used the fact $v_r\in\row{A_0(\epsilon)}\subset\row{A(\epsilon)}$, implying $R_{m+1}^*v_r = 0$. Similarly, we have
	\begin{equation}\label{eq:lb-vn}
		\begin{aligned}
			\nm{\mathcal{R}^{T}v_n}^2_{\bm V}={}&\nm{D_\epsilon^{-1}A(\epsilon)v_n}_{D_\epsilon}^2+\nm{R_{m+1}R_{m+1}^*v_n}^2\\
			={}&\nm{D_\epsilon^{-1}A(\epsilon)v_n}_{D_\epsilon}^2+\nm{v_n}^2\geq\nm{v_n}^2.
		\end{aligned}
	\end{equation}
	The cross term reads as follows
	\[
	\begin{aligned}
		\inner{\mathcal{R}\mathcal{R}^{T}v_r, v_n}
		=v_n^\top A^\top(\epsilon)	D_\epsilon^{-1}A(\epsilon)v_r+ \inner{R_{m+1}R_{m+1}^*v_r,v_n}=v_n^\top A^\top(\epsilon)	D_\epsilon^{-1}A(\epsilon)v_r.
	\end{aligned}
	\]
	Let consider the splitting $D_\epsilon = \diag{D_0,D_1}$, where $D_0=\diag{d_1,\cdots,d_{m_0}}$ and $D_1=\diag{d_{m_0+1},\cdots,d_m}$. Then by the identity $A_0(\epsilon) v_n = 0$, it follows that
	\[
	\begin{aligned}
		{}&			\left( \mathcal{R}\mathcal{R}^{T}v_r, v_n\right) 
		=			v_n^\top A^\top_1(\epsilon)	D_1^{-1}A_1(\epsilon) v_r=v_n^\top \left(E_1+A_1(0) \right)^\top	D_1^{-1}\left(E_1+A_1(0) \right)v_r\\
		={}&v_n^\top A^\top_1(0)D_1^{-1}A_1(0) v_r
		+v_n^\top E_1^\top D_1^{-1}E_1v_r+
		v_n^\top E_1^\top D_1^{-1}A_1(0)v_r
		+v_n^\top A_1^\top(0) D_1^{-1}E_1v_r,
	\end{aligned}
	\]
	where $E_1 = A_1(\epsilon)-A_1(0)$. 
	Notice that 
	$
	\snm{v_n^\top v_r}\leq\frac{1}{2}\left(\|v_r\|^2+\|v_n\|^2\right)=\frac{1}{2}\|v\|^2$ and 
	\[
	\begin{aligned}
		{}&	 v_n^\top E_1^\top D_1^{-1}E_1v_r+
		v_n^\top E_1^\top D_1^{-1}A_1(0)v_r
		+v_n^\top A_1^\top(0) D_1^{-1}E_1v_r\\
		\geq{}&-\frac{\snm{v_n^\top v_r}}{\eta_0}\left(\nm{E_1}^2+2\nm{E_1}\nm{A_1(0)}\right)
		\geq-\frac{\nm{v}^2}{2\eta_0}\left(\nm{E_1}^2+2\nm{E_1}\nm{A_1(0)}\right).
	\end{aligned}
	\]
	On the other hand, by \cref{def:nsp}, we have $A_1(0) = \Pi A_0(0)=\Pi(E_0+A_0(\epsilon))$, with $E_0=A_0(0)-A_0(\epsilon)$. It follows directly that 
	\[
	\begin{aligned}
		v_n^\top A^\top_1(0)D_1^{-1}A_1(0) v_r=	{}&	 v_n^\top (E_0+A_0(\epsilon))^\top\Pi^\top D_1^{-1}A_1(0)v_r
		=v_n^\top E_0^\top\Pi^\top D_1^{-1}A_1(0)v_r\\
		\geq{}&-\frac{\snm{v_n^\top v_r}}{\eta_0}\nm{E_0}\nm{A_1(0)}\nm{\Pi}\geq 
		-\frac{\nm{v}^2}{2\eta_0}\nm{E_0}\nm{A_1(0)}\nm{\Pi}.
	\end{aligned}
	\] 
	Consequently, collecting the above estimates gives $	\left( \mathcal{R}\mathcal{R}^{T}v_r, v_n\right) \geq 
	-\frac{C_1(\epsilon)}{2\eta_0}\nm{v}^2$,
	%\[
	%	\left( \mathcal{R}\mathcal{R}^{T}v_r, v_n\right) \geq 
	%	-\frac{\nm{v}^2}{2\eta_0}\left(\nm{E_1}^2+2\nm{E_1}\nm{A_1(0)}+\nm{E_0}\nm{A_1(0)}\nm{\Pi}\right),
	%\]
	which, together with \cref{eq:lb-vr,eq:lb-vn}, leads to
	\begin{align*}
		{}	&	\left( \mathcal{R}\mathcal{R}^{T}v, v \right)
		=  	\nm{\mathcal{R}^{T}v_r}^2_{\bm V}+2\left(\mathcal{R} \mathcal{R}^{T}v_r, v_n\right)+\nm{\mathcal{R}^{T}v_n}^2_{\bm V}\\
		\geq{}&\frac{1}{\eta_1}\min\left(\eta_1,\sigma_0(\epsilon)\right)\nm{v}^2
		-\frac{C_1(\epsilon)}{\eta_0}\nm{v}^2
		={}\left(\min\left(1,\sigma_0(\epsilon)/\eta_1\right)-C_1(\epsilon)/\eta_0\right)\nm{v}^2,
	\end{align*} 
	for any $v\in\row{A(\epsilon)}$. Plugging this into \cref{eq:C1} yields the desired estimate and thus completes the proof.
\end{proof}
\begin{rem}\label{rem:rho0}
	Let us explain more about the robust estimates given in \cref{thm:conv_robust}. By  \cref{assum:A0-A1} and \cref{lem:lambda-min-A0}, we have
	\begin{itemize}
		\item[(i)] $C_1(\epsilon)\to0$ as $\epsilon\to0+$;		
		\item[(ii)]	$\nm{A(\epsilon)}\leq C_2<\infty$ for all $\epsilon\in[0,\bar{\epsilon}]$; 		
		\item[(iii)] $\nm{D_\epsilon^{-1}A(\epsilon)}\leq\|A(\epsilon)\|/\eta_0
		%		\leq  \nm{A(\epsilon)}/\eta_0\leq \max_{0\leq\epsilon\leq\bar{\epsilon}}\nm{A(\epsilon)}/\eta_0
		$ for all $\epsilon\in[0,\bar{\epsilon}]$; 
		\item[(iv)] $\delta_{\max}(A(\epsilon))=\lambda_{\max}(A^\top(\epsilon)D^{-1}_\epsilon A(\epsilon))
		%		=\|D_\epsilon^{-1/2}A(\epsilon)\|^2
		\leq\|A(\epsilon)\|^2/\eta_0
		%		\leq  \nm{A(\epsilon)}/\eta_0\leq \max_{0\leq\epsilon\leq\bar{\epsilon}}\nm{A(\epsilon)}/\eta_0
		$ for all $\epsilon\in[0,\bar{\epsilon}]$; 	
		\item[(v)] $\sigma_0(\epsilon)>0$ for all $\epsilon\in[0,\epsilon_1]$ and $\lim_{\epsilon\to0+}\sigma_0(\epsilon) = \lambda_{\min}^+(A^\top_0(0)A_0(0))>0$. 
	\end{itemize}
	Hence, there exists some $\epsilon_2\in(0,\bar{\epsilon}]$ such that 
	\[
	\begin{aligned}
		\sigma_0(\epsilon)\geq{}& \frac{1}{2}\lambda_{\min}^+(A^\top_0(0)A_0(0)),\quad C_1(\epsilon)\leq \frac{\eta_0}{2}\min\left(1,\sigma_0(\epsilon)/\eta_1\right),
		%		C_0(\epsilon,\omega)\geq {}&4\omega(2-\omega)\left(2+  \omega  \lfloor \log_{2}(2m) \rfloor  + 2\omega\max\limits_{0\leq\epsilon\leq\bar{\epsilon}}\nm{A(\epsilon)}/\eta_0 \right)^{-2},
	\end{aligned}
	\]
	for all $\epsilon\in[0,\epsilon_2]$.
	%	 (i) $\nm{D_\epsilon^{-1}A(\epsilon)}\leq \nm{A(\epsilon)}/\eta_0\leq \max_{0\leq\epsilon\leq\bar{\epsilon}}\nm{A(\epsilon)}/\eta_0<\infty$ for all $\epsilon\in[0,\bar{\epsilon}]$; (ii) $C_1(\epsilon)\to0$ as $\epsilon\to0+$; and (iii)  $\sigma_0(\epsilon)>0$ for all $\epsilon\in[0,\epsilon_1]$. Hence, there exists some $\epsilon_2\in(0,\bar{\epsilon}]$ such that 
	%	\[
	%	\begin{aligned}
		%		\sigma_0(\epsilon)\geq{}& \frac{1}{2}\lambda_{\min}^+(A^\top_0(0)A_0(0)),\quad C_1(\epsilon)\leq \frac{\eta_0}{2}\min\left(1,\sigma_0(\epsilon)/\eta_1\right),
		%		%		C_0(\epsilon,\omega)\geq {}&4\omega(2-\omega)\left(2+  \omega  \lfloor \log_{2}(2m) \rfloor  + 2\omega\max\limits_{0\leq\epsilon\leq\bar{\epsilon}}\nm{A(\epsilon)}/\eta_0 \right)^{-2},
		%	\end{aligned}
	%	\]
	%	for all $\epsilon\in[0,\epsilon_2]$. 
	This implies a uniformly positive lower bound for the contraction factor \cref{eq:rho}: $	0<\rho_0(\omega)\leq\rho(\epsilon,\omega)\leq 1$, for all $\epsilon\in[0,\epsilon_2]$, where 
	\begin{equation}\label{eq:rhow}
		\rho_0(\omega)=\frac{\omega(2-\omega(1+C_2^2/\eta_0))\min\left(2\eta_1,\,\lambda_{\min}^+(A^\top_0(0)A_0(0))\right)}{\eta_1\left(2+  \omega  \lfloor \log_{2}(2m) \rfloor C_2^2/\eta_0 + 2\omega
			C_2/\sqrt{\eta_0}
			%			\nm{D_\epsilon^{-1}A(\epsilon)} 
			\right)^{2}}.
	\end{equation}
	%%%%%%%%%%%%%
	%	\[
	%			C_0(\epsilon,\omega) :={}	\frac{4\omega(2-\omega(1+\delta_{\max}(A(\epsilon))))\left( 	\min\left(1,\sigma_0(\epsilon)/\eta_1\right)-C_1(\epsilon)/\eta_0\right)}{\left(2+  \omega  \lfloor \log_{2}(2m) \rfloor \delta_{\max}(A(\epsilon)) + 2\omega\nm{D_\epsilon^{-1}A(\epsilon)} \right)^{2}}.
	%	\]
	%%%%%%%%%%%%%
\end{rem}

\subsection{A geometric illustration}

Consider a simple linear system $A(\epsilon)x=b$, where  $A(\epsilon)\in\R^{2\times 2}$ is given by \cref{eq:Ae1} with $\epsilon=1/5$. Geometrically, it corresponds to finding the intersection of two very close hyperplanes \(H_i = \{ x \in\R^2: a_i^\top x = b_i \}\) (\(i = 1,2\)). The convergence behaviors of Kaczmarz \cref{eq:kacz,eq:kak} are shown in \cref{simple_zigzag}, from which we observe the heavy zigzag phenomenon of Kaczmarz and the fast convergence of \cref{eq:kak}.

\begin{figure}[H]
	\centering
	\begin{subfigure}{0.45\textwidth}
		\centering
		\includegraphics[width=\linewidth]{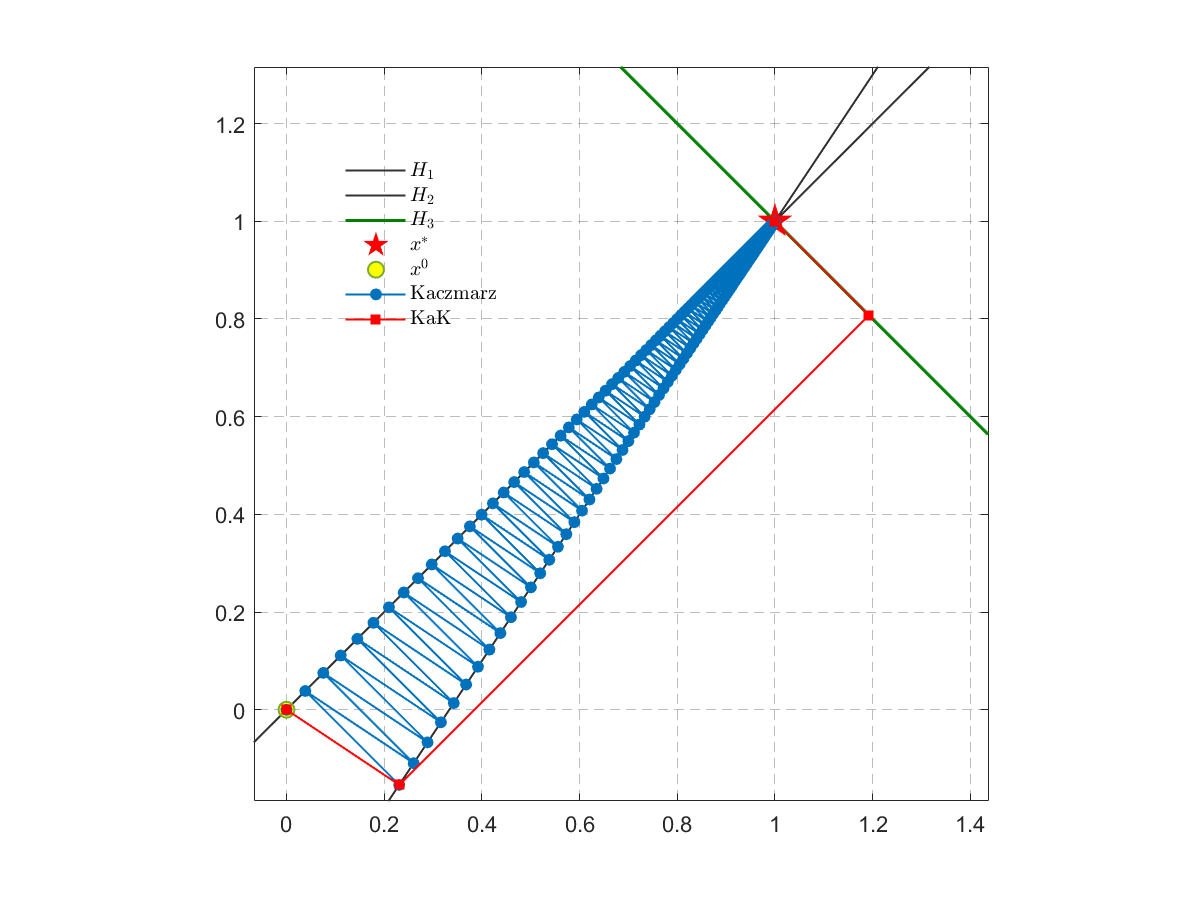}
	\end{subfigure}
	\hspace{0.3em}
	\begin{subfigure}{0.45\textwidth}
		\centering
		\includegraphics[width=\linewidth]{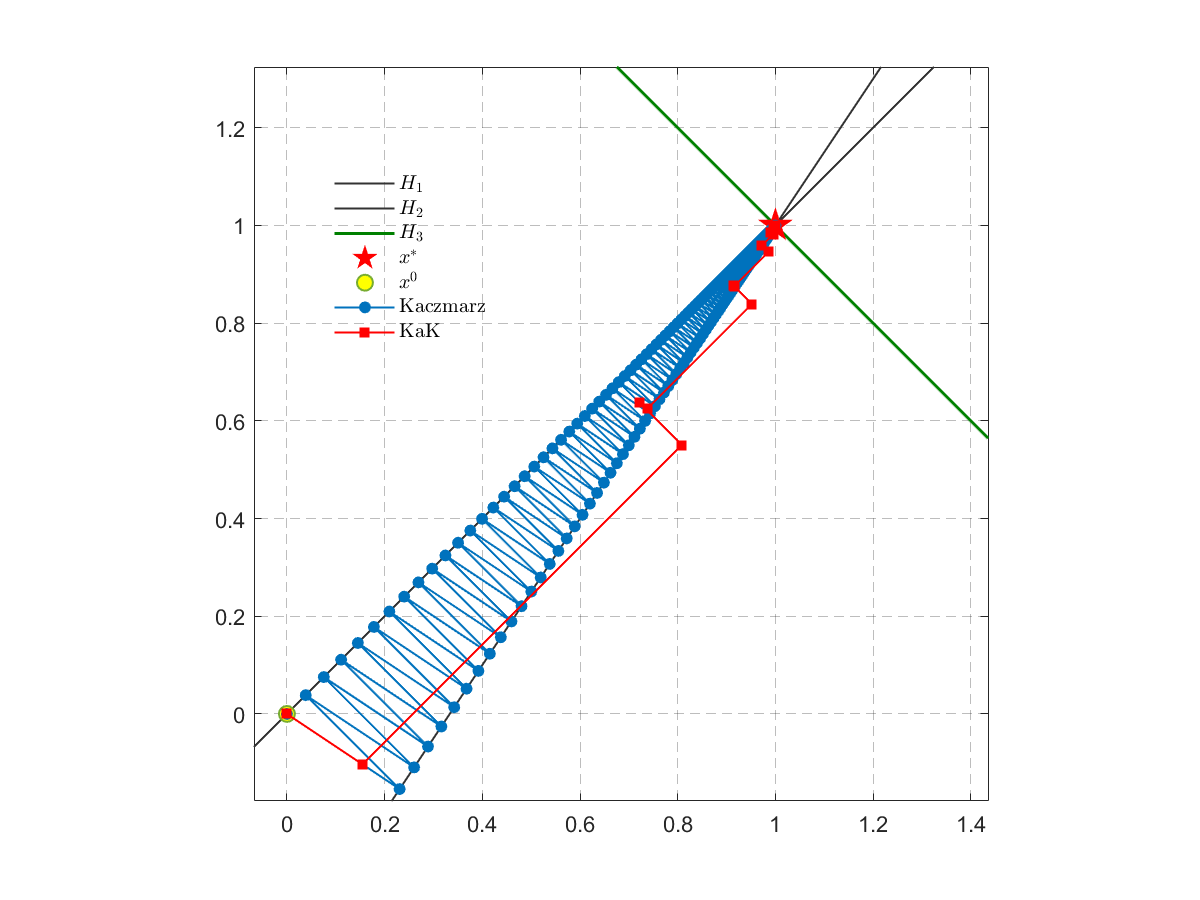}
	\end{subfigure}
	\caption{Illustrations of Kaczmarz and KaK for solving \cref{eq:Ax=b} with \cref{eq:Ae1}.  In both two pictures, we set $\omega=1$ for Kaczmarz; for KaK, $\omega$ is 1 for the left and $2/\delta_{max}(A(\epsilon))$ for the right.}
	\label{simple_zigzag}
\end{figure}
%\begin{figure}[H]
%	\centering
%	%	\begin{minipage}{0.5\textwidth}
	%		%		\centering
	%		\includegraphics[width=0.5\linewidth]{../fig/只加核的收敛行为1}
	%		\caption{Iterative illustrations of Kaczmarz \cref{eq:kacz,eq:kak}  for solving a simple linear system $A(\epsilon)x=b$, where $\epsilon=1/5$ and $A(\epsilon)\in\R^{2\times 2}$ is given by \cref{eq:Ae1}. \LH{Take $\omega\in(0,2/\delta_{\max}(A(\epsilon)))$. }\LH{KAK$\to$KaK. }\LH{Set the line of KaK in red. }\LH{Kaczmarz(with $\omega)\to$ Kaczmarz}}
	%		\label{simple_zigzag}
	%		%	\end{minipage}
%\end{figure}

For this simple but illustrative case, we can give an interesting geometrical explanation. By \cref{eg:ex1}, we have 
%$	\dappker{A(\epsilon)}=\myspan\{[1,-1]^\top\}$ and 
$\appker{A(\epsilon)}=\myspan\{\xi\}$ with $\xi = [1,1]^\top$. The extra step in \cref{eq:kak} reads as 
\[\small
\begin{aligned}
	x_{k+1}={}&v_{k,4}=v_{k,3}+\frac{\xi \xi^\top}{\|\xi\|^2}(x_{LS}-v_{k,3})
	=\mathop{\argmin}_{x\in\R^2}\left\{\frac{1}{2}\nm{x-v_{k,3}}^2:\,\,\xi^\top x=\xi^\top x_{LS}\right\}.
\end{aligned}
\]
Therefore, in each step of \cref{eq:kak}, the update of $x_{k+1}$ is nothing but the projection ($\omega=1$) or inexact projection ($\omega\neq 1$) of the output of Kaczmarz onto the hyperplane $H_3=\left\{x\in\R^2:\,\xi^\top x=\xi^\top x_{LS}\right\}$; see \cref{simple_zigzag}.

\section{Robust Coordinate Descent and its Acceleration}
\label{sec:rcd-acc}
In the last section, our \cref{eq:kak} demonstrates its effectiveness for nearly singular systems. However, it involves the priori information $x_{LS}=A^+b$. To save this limitation, we transform it into a CD  for solving the dual problem \cref{eq:dual-prob},
%\begin{equation}
%%	\label{eq:dual_problem}
%	\min_{y \in \mathbb{R}^m} g(y) = \frac{1}{2} \| A^\top(\epsilon) y \|^2 + b^\top y,
%\end{equation}
which is equivalent to
\begin{equation}\label{eq:AA'y-b}
	A(\epsilon)A^\top(\epsilon)y=-b.
\end{equation}
This is called the dual linear system, as a comparison with the primal one \cref{eq:Ax=b}.

As we all know, under the transformation $x_k = -A^\top(\epsilon) y_k$, the coordinate descent \cref{eq:cd-intro} for solving the dual problem \cref{eq:dual-prob} is equivalent to the classical Kaczmarz iteration \cref{eq:kacz} for the primal linear system \cref{eq:Ax=b}. Observing that Kaczmarz itself can be recovered from an SSC presentation (cf.\cref{eq:set-kacz-ssc}), it is natural to reformulate the coordinate descent \cref{eq:cd-intro} as an SSC form for solving the dual linear system \cref{eq:AA'y-b}. This motivates us developing and analyzing the corresponding robust coordinate descent methods via the SSC framework.
%Here, recall the standard CD iteration
%\begin{equation}
%\label{eq:cd_update}
%	y_{k+1} = y_k - \frac{1}{L_i} U_i \nabla g_i (y_k),
%\end{equation}
%where \(L_i\) is the Lipschitz constant, \(U_i\) the \(i\)-th coordinate unit vector, and \(\nabla g_i\) the gradient component. Substituting \(\nabla g(y) = A(\epsilon)A^\top(\epsilon) y + b\) gives:
%\begin{equation}
%	\label{eq:cd_specific}
%	y_{k+1} = y_k - \frac{1}{L_i} U_i U_i^\top (A(\epsilon)A^\top(\epsilon) y_k + b).
%\end{equation}
%Multiplying by \(-A^\top\) recovers the Kaczmarz iteration \cref{class_kacz}, establishing their equivalence. Introducing a relaxation parameter $\omega$  to 	\cref{eq:cd_specific} gives 
%\begin{equation}
%	\label{eq:cd_omega}
%	y_{k+1} = y_k - \omega\frac{1}{L_i} U_i U_i^\top (A(\epsilon)A(\epsilon)^\top y_k + b).
%\end{equation}
%where $L_i=\nm{A_{(i)}(\epsilon)}^2$. 
\subsection{Kernel-augmented CD}

Following the idea from \cref{eq:kak}, we consider the the following SSC setting:
\begin{equation}\label{eq:set-kacd-ssc}
	\renewcommand{\arraystretch}{1.2}
	\begin{array}{l}
		\textit{{\footnotesize$\bullet$}\,\, $V:=  \mathbb{R}^{m}$ with the Euclidean inner product $(x,y) :=x^\top y $;}\\
		\textit{{\footnotesize$\bullet$}\,\, $A:=A(\epsilon)A^\top(\epsilon) :V\to V$ and $f:= -b$;}\\		
		\text{{\footnotesize$\bullet$}\,\, $V= \sum_{i=1}^{m+1}V_i$ with $V_i:=\R$ for $1\leq i\leq m$ and $V_{m+1}:=\dappker{A(\epsilon)}$;}\\
		\text{{\footnotesize$\bullet$}\,\,  $R_i:=U_i:V_i\to V$ for $1\leq i\leq m$ and $R_{m+1} := \iota:V_{m+1}\to V$;}\\		
		\textit{{\footnotesize$\bullet$}\,\, $a_i(x_i,y_i)=d_ix_iy_i$ with $d_i=\nm{A_{(i)}(\epsilon)}^2>0$ for $1\leq i\leq m$};\\
		\textit{{\footnotesize$\bullet$}\,\,the inner product on $V_{m+1}$ is defined by $a_{m+1}(x,y) = x^\top A(\epsilon)A^\top(\epsilon) y$}.
	\end{array}
\end{equation}
%For notation simplicity, in the remainder of this section, we abbreviate $A(\epsilon)$ as $A$. Notice that the equivalence between 	\cref{eq:cd_omega} to \cref	{eq:kacz} under the transformation $x_k=-A^\top(\epsilon) y_k$ is important. Recall the approximate dual kernel in \cref{eq:app-dual-ker}:
%	\begin{equation*}
	%		\dappker{A(\epsilon)}:= \left\lbrace   y\in \R^{m}:A_{0}(\epsilon)A^{\top}(\epsilon)y=0\right\rbrace.
	%\end{equation*}
	%Letting $S_\epsilon = (\bar{w_1},\cdots,\bar{w_{M}})$, then $A^\top(\epsilon) S_{\epsilon}$ constitutes a basis for $\dappker{A(\epsilon)}$. Based on the robust SSC for the Kaczmarz method, we get the Kernel-Augmented CD method (KACD) naturally: 

	With this, it is easy to find $T_i = R_i^*A(\epsilon)A^\top(\epsilon) =   U_iA(\epsilon)A^\top(\epsilon)/d_i $ for $1\leq i \leq m$ and $ R_{m+1}^{*}:=\widehat{R}:\R^m\to \dappker{A(\epsilon)}$ is the orthogonal projection. This leads to an alternate robust SSC: given $y_k\in \R^m$, set $u_{k,1} = y_k$ and update $y_{k+1}=u_{k,m+2}$ by that
	\begin{equation}\label{eq:kacd}
		\tag{KaCD}
		\left\{
		\begin{aligned}
			{}&		u_{k,i+1} = 
			%		u_{k,i} +\omega R_iR_i^*(-A(\epsilon)A^\top(\epsilon) u_{k,i} -b)=
			u_{k,i} - \omega\frac{U_iU_i^\top}{\| A_{(i)}(\epsilon)\|^{2}} (A(\epsilon)A^\top(\epsilon) u_{k,i} +b),\quad\,1\leq i\leq m,\\
			{}&		u_{k,m+2}=u_{k,m+1}-\omega \widehat{R}(A(\epsilon)A^\top(\epsilon) u_{k,i} +b),
		\end{aligned}
		\right.
	\end{equation}
	which is called the {\it Kernel-augmented Coordinate Descent} (KaCD).
	%\[
	% \mathcal R^T=\mathcal R^\star =  D^{-1}A(\epsilon),\,\bm D = \diag{D^{-1}A(\epsilon)A^\top(\epsilon)},\, \bm U = \triu{D^{-1}A(\epsilon)A^\top(\epsilon)},
	%\]
	Note that we have the stationary iteration form
	\begin{equation}
		\label{eq:iter-Bssc-kacd}
		y_{k+1} = y_k - B_{\rm kacd}(\omega)(A(\epsilon)A^\top(\epsilon) y_k +b),
	\end{equation}
	where $B_{\rm kacd}(\omega):\R^m\to \R^m$ satisfies 
	\[
	I-B_{\rm kacd}(\omega) =(I-\omega R_{m+1}T_{m+1})(I-\omega R_mT_m)\cdots(I-\omega R_1T_1).
	\]
	%\begin{algorithm}[H]
	%		\caption{$y_{k+1} = \texttt{KaCD}(y_k,\omega,A,b,R)$}
	%	\label{algo:kacd}
	%	\begin{algorithmic}[1] 
		%		\REQUIRE  The relaxation parameter $	\omega\in(0,2)$ and the current iterate $y_k\in\R^m$.\\
		%		~~~~~~~The settings $A\in\R^{m\times n}$ and $b\in\R^m$ of the dual problem \cref{eq:dual_problem}.\\
		%~~~~~~~The orthogonal projection: $R:\R^m\to\dappker{A(\epsilon)}$.		
		%%		\FOR{$k=0,1,\cdots,K$}
		%		\STATE Set $v=		 y_k$.
		%		\FOR{$i=1,\cdots,m$}
		%		\STATE $v= v - \omega\frac{U_iU_i^\top  }{d_i} (A(\epsilon)A^\top(\epsilon) v+b)$.
		%		\ENDFOR 
		%		\STATE  $y_{k+1}=v-\omega R(A(\epsilon)A^\top(\epsilon) v+b)$. 
		%%		\ENDFOR
		%%\ENSURE 
		%	\end{algorithmic}
	%\end{algorithm}
	
	According to \cref{sec:sym-ssc}, we also have the symmetrized version of \cref{eq:kacd}:
	\begin{equation}\label{eq:symkacd}
		\tag{SymKaCD}
		z_{k+1} = z_{k} -\bar {B}_{\rm kacd}(\omega) \left(A(\epsilon)A^{\top}(\epsilon) z_{k}+b\right),
	\end{equation}
	where $ \bar {B}_{\rm kacd}(\omega)=B_{\rm kacd}(\omega)+B'_{\rm kacd}(\omega)-B'_{\rm kacd}(\omega)A(\epsilon)A^\top(\epsilon) B_{\rm kacd}(\omega)$ is the symmetrized iterator. A detailed one step implementable version of \cref{eq:symkacd} is given in \cref{algo:symkacd}, and the convergence analysis is summarized in \cref{thm:conv-kacd}, which provides robust convergence rates of since by \cref{rem:rho0} there exist $\rho_0(\omega)>0$ and $\epsilon_2\in(0,\bar{\epsilon}]$ such that $0<\rho_0(\omega)\leq \rho(\epsilon,\omega)\leq 1$ for all $\epsilon\in[0,\epsilon_2]$.
	\begin{algorithm}[H]
		\caption{$z_{k+1} = \texttt{SymKaCD}(z_k,\omega,A(\epsilon),b,\widehat{R})$}
		\label{algo:symkacd}
		\begin{algorithmic}[1] 
			\REQUIRE  $	\omega\in(0,2/(1+\delta_{\max}(A(\epsilon))))$ and $z_k\in\R^m$.\\
			%			~~~~~~~~~The setting of the dual problem \cref{eq:dual_problem}: $A(\epsilon)\in\R^{m\times n}$ and $b\in\R^m$.\\
			~~~~~~~~~The orthogonal projection: $\widehat{R}:\R^m\to\dappker{A(\epsilon)}$.		
			%		\FOR{$k=0,1,\cdots,K$}
			\STATE Set $v=		 z_k$.
			\FOR{$i=1,\cdots,m$}
			\STATE $v= v - \omega\frac{U_iU_i^\top  }{d_i} (A(\epsilon)A^\top(\epsilon) v+b)$.
			\ENDFOR 
			\STATE  $v=v-\omega \widehat{R}(A(\epsilon)A^\top(\epsilon) v+b)$. 
			\STATE  $v=v-\omega \widehat{R}(A(\epsilon)A^\top(\epsilon) v+b)$. 		
			\FOR{$i=m,\cdots,1$}
			\STATE $v= v - \omega\frac{U_iU_i^\top  }{d_i} (A(\epsilon)A^\top(\epsilon) v+b)$.
			\ENDFOR 
			\STATE  $z_{k+1}=v$. 		
		\end{algorithmic}
	\end{algorithm}
	%\begin{algorithm}[H]
	%	\caption{A symmetric Kernel-Augmented CD method for solving \cref{eq:dual_problem}}
	%	\label{algo:kacd-sym}
	%	\begin{algorithmic}[1] 
		%		\REQUIRE  Approximate dual kernel space: $\bar{W}_\epsilon={\rm span}\{\bar{w}_j\}_{j=1}^M$.\\		
		%		~~~~~~~Initial values: $	y_0\in\R^m$.\\		
		%		~~~~~~~Maximum iterations: $ K\in\mathbb N$.	\\
		%		\FOR{$k=0,1,\cdots,K$}
		%		\STATE Set $y_{0}=v$.
		%		\STATE Update $v=  v- S_\epsilon\left(S_\epsilon^\top{A(\epsilon)A^\top(\epsilon)} S_\epsilon \right)^{-1}S_\epsilon^\top (b+A(\epsilon)A^\top(\epsilon) v)$. 		
		%		\FOR{$i=1,\cdots,m$}
		%		\STATE Set $y_{k+1}^{1}=v$
		%		\STATE Update $v=  v - \frac{U_iU_i^\top  }{d_i} (b+A(\epsilon)A^\top(\epsilon) v)$.
		%		\ENDFOR
		%		\
		%		\STATE Update $v= v - S_\epsilon\left(S_\epsilon^\top {A(\epsilon)A^\top(\epsilon)}S_\epsilon \right)^{-1}S_\epsilon^\top (b+A(\epsilon)A^\top(\epsilon) v)$.
		%		\STATE $y_{k+1}=v.$ 
		%		\ENDFOR
		%	\end{algorithmic}
	%\end{algorithm}
	
	\begin{thm}\label{thm:conv-kacd}
		Assume that $A(\epsilon)$  has full of row rank and the relaxation parameter satisfies $	\omega\in(0,2/(1+\delta_{\max}(A(\epsilon))))$. Let $\{x_k\}_{k\geq0}$ be generated by \cref{eq:kak} and $\{y_k\}_{k\geq0}$ by \cref{eq:kacd} with $x_0=-A^\top(\epsilon) y_0$. Then we have the following.
		\begin{itemize}
			\item[(i)] \cref{eq:kacd} with the setting \cref{eq:set-kacd-ssc} is equivalent to \cref{eq:kak} with the setting \cref{eq:set-kak-ssc} in the sense that $x_k=-A^\top(\epsilon) y_k$ for all $k\geq0$.
			\item [(ii)]  $	
			B_{\rm kak}(\omega) = {}A^\top(\epsilon)  B_{\rm kacd}(\omega)A(\epsilon)
			$ and 
			\begin{equation}\label{eq:nm-iter-kacd}\small
				\begin{aligned}
					{}&	\nm{I-\bar{B}_{\rm kacd}A(\epsilon)A^\top(\epsilon)}_{A(\epsilon)A^\top(\epsilon)}=	\nm{I-B_{\rm kacd}(\omega)A(\epsilon)A^\top(\epsilon)}_{A(\epsilon)A^\top(\epsilon)}^2\\
					={}&		\nm{I-		B_{\rm kak}(\omega) }^2_{\row{A(\epsilon)}\to\row{A(\epsilon)}}\leq 1-\rho(\epsilon,\omega),
				\end{aligned}
			\end{equation}
			where $\rho(\epsilon,\omega)$ is defined by \cref{eq:rho}. 
			\item[(iii)] Moreover, for \cref{eq:kacd,eq:symkacd}, we have
			\[
			\begin{aligned}
				\nm{z_k-y^*}_{A(\epsilon)A^\top(\epsilon)}
				\leq {}& \left(1-\rho(\epsilon,\omega)\right) ^k\nm{z_0-y^*}_{A(\epsilon)A^\top(\epsilon)},\\
				\nm{x_k-x_{LS}}^2=	\nm{y_k-y^*}_{A(\epsilon)A^\top(\epsilon)}^2
				\leq {}& \left(1-\rho(\epsilon,\omega)\right) ^k\nm{y_0-y^*}_{A(\epsilon)A^\top(\epsilon)}^2,
			\end{aligned}
			\]
			where $y^*=-(A(\epsilon)A^\top(\epsilon))^{-1}b$ is the unique solution to the linear system \cref{eq:AA'y-b}.
		\end{itemize}
	\end{thm}
	
	\begin{proof}
		Let us first show the relation $R = A^\top(\epsilon)\widehat{R}A(\epsilon)$, where $R:\row{A(\epsilon)}\to\appker{A(\epsilon)}$ and $\widehat{R}:\R^m\to\dappker{A(\epsilon)}$ are orthogonal projections. 
		By \cref{def:ak}, $\appker{A(\epsilon)}$ is the image of $	\dappker{A(\epsilon)}$ under the linear mapping $A^\top(\epsilon):\R^m\to\R^n$. Since \( A(\epsilon) \) has full row rank, 
		it follows that \( \dim \appker{A(\epsilon)} = \dim 		\dappker{A(\epsilon)}=r_1\). Let the columns of $S\in\R^{m\times r_1}$ form  a basis of $\dappker{A(\epsilon)}$, then we have  $\widehat{R}=S\left(S^\top A(\epsilon) A^\top(\epsilon) S\right)^{-1}S^\top$. Note that the columns of $W = A^\top(\epsilon) S\in\R^{n\times r_1}$ form a basis of $\appker{A(\epsilon)} $, and thus 
		\[
		R = W\left(W^\top W\right)^{-1}W^\top = A^\top(\epsilon) S\left(S^\top A(\epsilon) A^\top(\epsilon) S\right)^{-1}S^\top A(\epsilon)=A^\top(\epsilon)\widehat{R}A(\epsilon).
		\] 
		In the following, let us verify (i)-(iii) one by one. 
		
		Suppose that $x_k = -A^\top(\epsilon)y_k$, then by \cref{eq:kacd}, we have
		\[
		\left\{
		\begin{aligned}
			-A^\top(\epsilon) u_{k,1} ={}& -A^\top(\epsilon) y_k = x_k=v_{k,1},\\
			-A^\top(\epsilon)	u_{k,i+1} = {}& 
			-A^\top(\epsilon) u_{k,i} -  \omega A^\top(\epsilon)\frac{U_iU_i^\top}{\| A_{(i)}(\epsilon)\|^{2}} (-A(\epsilon)A^\top(\epsilon) u_{k,i} -b)\\
			={}&	-A^\top(\epsilon) u_{k,i} +\omega A_{(i)}^\top(\epsilon)\frac{b_{i} +A_{(i)}(\epsilon) A^\top(\epsilon) u_{k,i} }{\| A_{(i)}(\epsilon)\|^{2}} ,
		\end{aligned}
		\right.
		\]
		for all $1\leq i\leq m$. Compared with  \cref{eq:kak}, we conclude immediately that $v_{k,i} =-A^\top(\epsilon)u_{k,i} $ for all $2\leq i\leq m+1$. Thus, using \cref{eq:kacd} again, we obtain (noticing that $Ax_{LS}=b$)
		\[
		\begin{aligned}
			{}&		-A^\top(\epsilon) u_{k,m+2}=-A^\top(\epsilon) u_{k,m+1}+\omega A^\top(\epsilon)\widehat{R}(A(\epsilon)A^\top(\epsilon) u_{k,m+1} +b)\\
			={}&v_{k,m+1}+\omega A^\top(\epsilon)\widehat{R}A(\epsilon)(x_{LS}- v_{k,m+1} )=v_{k,m+1}+\omega R(x_{LS}- v_{k,m+1} ),
		\end{aligned}
		\]
		which implies $x_{k+1}=v_{k,m+2}=	-A^\top(\epsilon) u_{k,m+2}=-A^\top(\epsilon) y_{k+1}$. This verifies (i). 
		
		In view of \cref{eq:iter-Bssc-kacz,eq:iter-Bssc-kacd}, it is easy to obtain $	
		B_{\rm kak}(\omega) = {}A^\top(\epsilon)  B_{\rm kacd}(\omega)A(\epsilon)
		$. Thanks to the \cref{eq:xz-id}, we claim that 
		\[
		\begin{aligned}
			\nm{I-\bar{B}_{\rm kacd}A(\epsilon)A^\top(\epsilon)}_{A(\epsilon)A^\top(\epsilon)}={}&\nm{I-B_{\rm kacd}(\omega)A(\epsilon)A^\top(\epsilon)}_{A(\epsilon)A^\top(\epsilon)}^2,
		\end{aligned}
		\]
		and it follows from \cref{thm:conv_robust} that
		\[\small
		\begin{aligned}
			{}&	\nm{I-B_{\rm kacd}(\omega)A(\epsilon)A^\top(\epsilon)}_{A(\epsilon)A^\top(\epsilon)}^2=\sup_{u\in\R^m}\frac{\nm{(I-B_{\rm kacd}(\omega)A(\epsilon)A^\top(\epsilon))u}_{A(\epsilon)A^\top(\epsilon)}^2 }{\nm{u}^2_{A(\epsilon)A^\top(\epsilon)}}\\	
			={}&\sup_{u\in\R^m}\frac{\nm{(I-A^\top(\epsilon)B_{\rm kacd}(\omega)A(\epsilon))A^\top(\epsilon)u}^2}{\nm{A^\top(\epsilon)u}^2}
			={}\sup_{v\in\row{A(\epsilon)}}\frac{\nm{(I-A^\top(\epsilon)B_{\rm kacd}(\omega)A(\epsilon))v}^2}{\nm{v}^2}\\		
			={}&\sup_{v\in\row{A(\epsilon)}}\frac{\nm{(I-B_{\rm kak}(\omega))v}^2}{\nm{v}^2}
			={}	\nm{I-		B_{\rm kak}(\omega) }^2_{\row{A(\epsilon)}\to \row{A(\epsilon)}}\leq 1-\rho(\epsilon,\omega),
		\end{aligned}
		\]
		which proves \cref{eq:nm-iter-kacd} and verifies (ii). This together with the stationary iteration forms \cref{eq:iter-Bssc-kacd,eq:symkacd} leads to (iii) and concludes the proof.
	\end{proof}

	%\iffalse
	%\begin{table}[htbp]
	%	\caption{When $A$ is full of row rank, the equivalence between Kacz for the primal problem and CD for the dual problem.}
	%	\begin{tabular}{c c c} % 表格定义
		%		\toprule
		%		\centering
		%		
		%		Kacz for $Ax=b$ & & CD for $min g(y)$ \\
		%		\midrule
		%		SSC & $\xleftrightarrow{}$ & SSC \\
		%		robust SSC: $W_\epsilon = V_{m+1}$ & $\xleftrightarrow{}$ & robust SSC: $\bar W_\epsilon = V_{m+1}$ \\
		%		robust SSC: $V_{m+j} = {\rm span}\{w_j=A^\top(\epsilon) \bar w_j\}$ & $\xleftrightarrow{}$ & robust SSC: $V_{m+j} = {\rm span}\{\bar w_j\}$ \\
		%		\bottomrule
		%	\end{tabular}
	%\end{table}\fi
	
	\subsection{Acceleration}
	\label{sec:AccK}
	We now consider an accelerated variant of \cref{algo:symkacd} for the dual problem \cref{eq:dual-prob}. The basic idea is to combine \cref{algo:symkacd}  with a predictor-corrector scheme  \cite[Eq.(83)]{Luo2022}. The resulted method is called the {\it Kernel-augmented Accelerated Coordinate Descent} (KaACD) and has been presented in \cref{algo:kaacd}.
		\begin{algorithm}[H] 
			\caption{KaACD for solving \cref{eq:dual-prob}}
			\label{algo:kaacd}
			\begin{algorithmic}[1] 
				\REQUIRE   The relaxation parameter $	\omega\in(0,2/(1+\delta_{\max}(A(\epsilon))))$.\\
				~~~~~~~~~Initial values: $	y_0,\,v_0\in\R^m,\,\gamma_0>0$.\\
				~~~~~~~~~The orthogonal projection: $\widehat{R}:\R^m\to\dappker{A(\epsilon)}$.		\\							
				%			~~~~~~~The setting of the dual problem \cref{eq:dual_problem}: $A(\epsilon)\in\R^{m\times n}$ and $b\in\R^m$.\\
				%		\FOR{$k=0,1,\cdots,K$}
				%		\REQUIRE  Approximate dual kernel space: $\bar{W}_\epsilon={\rm span}\{\bar{w}_j\}_{j=1}^M$.\\		
				~~~~~~~~~Convexity parameter: $	\rho \in[0,\rho(\epsilon,\omega)]$, where $\rho(\epsilon,\omega)$ is defined by  \cref{eq:rho}.\\		
				\FOR{$k=0,1,\cdots$}
				\STATE Compute $\alpha_k=(\gamma_k+\sqrt{\gamma_k^2+4\gamma_k})/2$.\label{algo:ak}
				\STATE Update $\gamma_{k+1}=(\gamma_k+\rho \alpha_k)/(1+\alpha_k)$.
				\STATE Compute $z_k = (y_k+\alpha_kv_k)/(1+\alpha_k)$.\label{algo:zk}
				\STATE Update $z_{k+1} = \texttt{SymKaCD}(z_k,\omega,A(\epsilon),b,\widehat{R})\,\,\left(\text{call \cref{algo:symkacd}}\right)$.
				%		\STATE Set $z = z_k$.
				%		\STATE 	$z = z + \bar {B}_{\rm kacd} \left(-b - AA^{\top}z\right)$
				%		\STATE Update $z_{k+1}=  z$. 
				\STATE Update $v_{k+1} = (\gamma_kv_k+\rho \alpha_k z_k+\alpha_k(z_{k+1}-z_k))/(\gamma_k+\rho \alpha_k)$.\label{algo:vk1}
				\STATE Update $y_{k+1} = (y_k+\alpha_k v_{k+1})/(1+\alpha_k)$.\label{algo:yk1}
				\ENDFOR						
			\end{algorithmic}
		\end{algorithm}

		Notably, thanks to \cref{eq:symkacd}, we have 
		\begin{equation}\label{eq:zk1}
			z_{k+1} = z_{k} -\bar {B}_{\rm kacd}(\omega) \left(A(\epsilon)A^{\top}(\epsilon) z_{k}+b\right)=z_{k}-\bar {B}_{\rm kacd}(\omega)\nabla g(z_k),
		\end{equation}
		which is actually a preconditioned gradient descent step. This motivates us treating \cref{algo:kaacd} as a preconditioned accelerated gradient method. Moreover, by \cref{coro2.1}, \cref{rem:rho0}, and \cref{thm:conv-kacd}, we have
		%In this formulation, the linear operator $B$ corresponds precisely to the matrix $\bar{B}_{\mathrm{kacd}}$ constructed in Algorithm~\ref{algo:kaacd}, and we have $\bar {B}_{\rm kacd}=B=B'$. Thanks to \cref{thm:conv_robust}, we see that
		%\begin{equation*}
		%	\nm{I-B_{\rm rkacz}}^2 = 		\nm{I-\bar{B}_{\rm  rkacz}}\leq1-	\sigma,
		%\end{equation*} where $	\sigma = \frac{\lambda_{\min }\left(A_{0}^{\top}(0) D_{0}^{-1} A_{0}(0)\right)+C \epsilon}{1+\left\lfloor\log _{2}(2 m)\right\rfloor+2 \sqrt{d_{1} / d_{0}}}$. Hence, according to \cref{coro2.1}, we have
		%\iffalse $B=\bar{B}_{\rm kacd} = \tilde{B}A(\epsilon)A^\top(\epsilon)$  and $B=B'$.  Then
		%\[(Bu,v)_{A(\epsilon)A^\top(\epsilon)}=(u,Bv)_{A(\epsilon)A^\top(\epsilon)},\quad\forall u,v\in \mathbb{R}^m \]which means
		%\[v^\top A(\epsilon)A^\top(\epsilon)\tilde{B}A(\epsilon)A^\top(\epsilon) u = v^\top A(\epsilon)A^\top(\epsilon)\tilde{B}^\top A(\epsilon)A^\top(\epsilon) u,\quad\forall u,v\in \mathbb{R}^m \]Thus we get $\tilde{B}^\top = \tilde{B}.$ Meanwhile, notice that $B$ is SPD with inner product induced by the matrix $A(\epsilon)A^\top(\epsilon)$:
		%\[
		%(\tilde{B}A(\epsilon)A^\top(\epsilon) v,v)_{A(\epsilon)A^\top(\epsilon)}>0\quad\forall v\in \mathbb{R}^m \backslash \{0\}. \] Let $H= A(\epsilon)A^\top(\epsilon) \tilde{B}A(\epsilon)A^\top(\epsilon)$, then $v^\top \tilde{B}v=v^\top (A(\epsilon)A^\top(\epsilon))^{-1}H(A(\epsilon)A^\top(\epsilon))^{-1} v >0$ for all $v\in \mathbb{R}^m \backslash \{0\}$. Therefore, $\tilde{B}^{-1}$ is SPD with ​​natural inner product​​.\fi 
		\begin{equation}\label{eq:mu-L-g}
			\rho(\epsilon,\omega)\dual{\bar {B}^{-1}_{\rm kacd}(\omega)y,y}\leq \dual{A(\epsilon)A^\top(\epsilon) y,y}\leq  \dual{\bar {B}^{-1}_{\rm kacd}(\omega)y,y}\quad\forall\,y\in\R^m,
		\end{equation}
		which is equivalent to 
		\begin{align}
			\label{eq:sc-L}
			g(y)\leq{}& g(z)+\dual{\nabla g(z),y-z}+\frac{1}{2}\nm{y-z}^2_{\bar {B}^{-1}_{\rm kacd}(\omega)}\quad\forall\,y,\,z\in\R^m,\\
			\label{eq:sc-g}
			g(y)\geq{}& g(z)+\dual{\nabla g(z),y-z}+\frac{\rho(\epsilon,\omega)}{2}\nm{y-z}^2_{\bar {B}^{-1}_{\rm kacd}(\omega)}\quad\forall\,y,\,z\in\R^m.
		\end{align}
		This key insight is crucial for the subsequent convergence rate analysis. We emphasize that by \cref{rem:rho0}, there exist $\rho_0(\omega)>0$ and $\epsilon_2\in(0,\bar{\epsilon}]$ such that $0<\rho_0(\omega)\leq \rho(\epsilon,\omega)\leq 1$ for all $\epsilon\in[0,\epsilon_2]$.

		\begin{thm}\label{thm:conv-ey0-ode-PNAG}
			Assume $A(\epsilon)$ has full row rank.
			Let $\{y_k\}_{k\geq 0}$ and $\{v_k\}_{k\geq 0}$ be generated by \cref{algo:kaacd}. Define the nonnegative function \begin{equation}\label{eq:Lxk}
				\mathcal L_k :={} g(y_k)-g(y^*)+\frac{\gamma_k}2\nm{v_{k}-y^*}_{\bar {B}^{-1}_{\rm kacd}(\omega)}^2.
			\end{equation} 
			Then we have the contraction property
			\begin{equation}\label{eq:Pdecaz-Lk-ey0}
				\mathcal{L}_{k+1} - \mathcal{L}_{k} \leq -\alpha_k \mathcal{L}_{k+1}.
			\end{equation}
			In addition, if $\gamma_0\geq\rho$, then
			\begin{equation}\label{eq:lk-rate}
				\mathcal{L}_k 
				\leq  \mathcal{L}_0\times\min\left\{\frac{4}{\gamma_0k^2},\,\left(1+\sqrt{\rho}\right)^{-k}\right\}.
			\end{equation}
			Consequently, we obtain
			\begin{equation}\label{eq:final-rate}
				%		\small
				\nm{x_k-x_{LS}}^2= \|y_k - y^*\|_{A(\epsilon)A^\top(\epsilon)}^2 \leq2 \mathcal{L}_0\times\min\left\{\frac{4}{\gamma_0k^2},\,\left(1+\sqrt{\rho}\right)^{-k}\right\},
			\end{equation}
			where $x_k = -A^\top(\epsilon) y_k$ denotes the sequence for the primal system \cref{eq:Ax=b}.
		\end{thm}
		
		\begin{proof}
			Let us start with the difference
			\[
			\begin{aligned}
				\mathcal L_{k+1}-\mathcal L_{k}
				={}& g(y_{k+1})-g(y_k)+
				\frac{\alpha_k}{2}(\rho - \gamma_{k+1})
				\nm{v_{k+1}-y^*}^2_{\bar {B}^{-1}_{\rm kacd}(\omega)}\\
				{}&	+\gamma_k
				\dual{v_{k+1}-v_k, v_{k+1} -y^*}_{\bar {B}^{-1}_{\rm kacd}(\omega)} -
				\frac{\gamma_k}2
				\nm{v_{k+1}-v_k}^2_{\bar {B}^{-1}_{\rm kacd}(\omega)}.
			\end{aligned}
			\]
			From \cref{eq:zk1} and the Line \ref{algo:vk1} of \cref{algo:kaacd}, it is clear that 
			\[
			\gamma_k\frac{v_{k+1}-v_{k}}{\alpha_k} = \rho (z_{k} - v_{k+1}) - \bar {B}_{\rm kacd}(\omega)\nabla g(z_k),
			\]
			which implies 
			\[\small
			\begin{aligned}
				{}&\gamma_k\dual{v_{k+1}-v_k, v_{k+1} -y^*}_{\bar {B}^{-1}_{\rm kacd}(\omega)} \\
				= {}&
				\rho \alpha_k\dual{z_k - v_{k+1}, v_{k+1} - y^{*}}_{\bar {B}^{-1}_{\rm kacd}(\omega)}
				-
				\alpha_k \dual{\nabla g(z_k), v_{k+1} - y^{*}}\\
				={}&	\rho \alpha_k\dual{z_k - v_{k+1}, v_{k+1} - y^{*}}_{\bar {B}^{-1}_{\rm kacd}(\omega)}
				-\alpha_k \dual{ \nabla g(z_k), z_{k} - y^{*}}-\alpha_k \dual{ \nabla g(z_k), v_{k+1} - z_k}.
			\end{aligned}
			\]
			We further split the first term as follows:
			\begin{equation*}
				\begin{aligned}
					{}&2\dual{z_k - v_{k+1}, v_{k+1} - y^{*}}_{\bar {B}^{-1}_{\rm kacd}(\omega)} \\
					={}& \left\|z_k-y^{*}\right\|^{2}_{\bar {B}^{-1}_{\rm kacd}(\omega)}
					-\|z_k-v_{k+1}\|^{2}_{\bar {B}^{-1}_{\rm kacd}(\omega)}
					-\left\|v_{k+1}-y^{*}\right\|^{2}_{\bar {B}^{-1}_{\rm kacd}(\omega)}.
				\end{aligned}
			\end{equation*}
			Then it follows that
			\[
			\begin{aligned}
				\mathcal L_{k+1}-\mathcal L_{k}
				={}& g(y_{k+1})-g(y_k)		- \alpha_k\dual{ \nabla g(z_k), v_{k+1} - z_{k}} -
				\frac{\alpha_k\gamma_{k+1}}{2}
				\nm{v_{k+1}-y^*}^2_{\bar {B}^{-1}_{\rm kacd}(\omega)}\\
				{}&	 -
				\frac{\gamma_k}2
				\nm{v_{k+1}-v_k}^2_{\bar {B}^{-1}_{\rm kacd}(\omega)}	-\frac{\rho  \alpha_k}{2}\|z_k-v_{k+1}\|^{2}_{\bar {B}^{-1}_{\rm kacd}(\omega)} 
				\\
				{}&\quad+\frac{\rho \alpha_k}{2}\left\|z_k-y^{*}\right\|^{2}_{\bar {B}^{-1}_{\rm kacd}(\omega)}- \alpha_k\dual{ \nabla g(z_k), z_k - y^{*}} .
			\end{aligned}
			\]
			Since $\rho \leq\rho(\epsilon,\omega)$, in view of \cref{eq:sc-g}, we have
			\[
			\begin{aligned}
				{}&\frac{\rho \alpha_k}{2}\left\|z_k-y^{*}\right\|^{2}_{\bar {B}^{-1}_{\rm kacd}(\omega)}- \alpha_k\dual{ \nabla g(z_k), z_k - y^{*}}\\
				\leq{}&\frac{\alpha_k(\rho -\rho(\epsilon,\omega))}{2}\nm{z_k-y^*}^2_{\bar {B}^{-1}_{\rm kacd}(\omega)}-\alpha_k\left(g(z_k)-g(y^*)\right)\\
				\leq {}&-\alpha_k \left(g(y_{k+1})-g(y^*)\right)
				+\alpha_k\left(g(y_{k+1})-g(z_k)\right).
			\end{aligned}
			\]
			Hence, this gives 
			\[
			\begin{aligned}
				\mathcal L_{k+1}-\mathcal L_{k}
				\leq {}&-\alpha_k \left(g(y_{k+1})-g(y^*)\right)-
				\frac{\alpha_k\gamma_{k+1}}{2}
				\nm{v_{k+1}-y^*}^2_{\bar {B}^{-1}_{\rm kacd}(\omega)}\\
				{}&\quad +g(y_{k+1})-g(y_k)		- \alpha_k\dual{ \nabla g(z_k), v_{k+1} - z_{k}}\\
				{}&	\qquad +\alpha_k\left(g(y_{k+1})-g(z_k)\right)-
				\frac{\gamma_k}2
				\nm{v_{k+1}-v_k}^2_{\bar {B}^{-1}_{\rm kacd}(\omega)}	 \\
				={}&-\alpha_k  {\mathcal L}_{k+1}
				-	\frac{\gamma_k}2\nm{v_{k+1}-v_k}^2_{\bar {B}^{-1}_{\rm kacd}(\omega)}- \alpha_k\dual{ \nabla g(z_k), v_{k+1} - z_{k}}
				\\
				{}&\quad +g(y_{k+1})-g(y_k)		+\alpha_k\left(g(y_{k+1})-g(z_k)\right).
			\end{aligned}
			\]
			Thanks to \cref{eq:sc-L}, we have 
			\[
			\begin{aligned}
				g(y_{k+1}) \leq{}& g(z_k) + \langle \nabla g(z_k), y_{k+1} - z_k \rangle + \frac{1}{2} \| y_{k+1} - z_k \|_{\bar {B}^{-1}_{\rm kacd}(\omega)}^2.
			\end{aligned}
			\]
			The update rules for $y_{k+1}$ and $z_k$ (cf.the Lines \ref{algo:zk} and \ref{algo:yk1} of \cref{algo:kaacd}) are
			\[
			z_k = \frac{y_k + \alpha_k v_k}{1+\alpha_k}, \quad y_{k+1} = \frac{y_k + \alpha_k v_{k+1}}{1+\alpha_k},
			\]
			which yields that $		y_{k+1} - z_k = \frac{\alpha_k}{1+\alpha_k} (v_{k+1} - v_k)$. Hence, we obtain
			\[
			g(y_{k+1}) - g(z_k) \le \frac{\alpha_k}{1+\alpha_k} \langle \nabla g(z_k), v_{k+1}-v_k \rangle + \frac{1}{2}\left(\frac{\alpha_k}{1+\alpha_k}\right)^2 \| v_{k+1} - v_k \|_{\bar {B}^{-1}_{\rm kacd}(\omega)}^2,
			\]
			and by the convexity of $g$, it follows that $	g(z_{k}) \leq{} g(y_k) + \langle \nabla g(z_k), z_{k} - y_k \rangle  			=
			g(y_k) +\alpha_k \langle \nabla g(z_k), v_{k} - z_k \rangle $.
			Consequently, we arrive at
			\[\small
			\begin{aligned}
				\mathcal L_{k+1}-\mathcal L_{k}
				\leq {}&-\alpha_k  {\mathcal L}_{k+1}
				-	\frac{\gamma_k}2\nm{v_{k+1}-v_k}^2_{\bar {B}^{-1}_{\rm kacd}(\omega)}- \alpha_k\dual{ \nabla g(z_k), v_{k+1} - z_{k}}
				\\
				{}&\quad +g(z_{k})-g(y_k)+ \alpha_k \langle \nabla g(z_k), v_{k+1}-v_k \rangle + \frac{1}{2}\frac{\alpha_k^2}{1+\alpha_k} \| v_{k+1} - v_k \|_{\bar {B}^{-1}_{\rm kacd}(\omega)}^2\\
				\leq{}& -\alpha_k  {\mathcal L}_{k+1}
				+	\frac{1}{2(1+\alpha_k)}\left(\alpha_k^2-\gamma_k(1+\alpha_k)\right)\nm{v_{k+1}-v_k}^2_{\bar {B}^{-1}_{\rm kacd}(\omega)}.
			\end{aligned}
			\]
			By the Line \ref{algo:ak} of \cref{algo:kaacd}, it is clear that $\alpha_k^2=\gamma_k(1+\alpha_k)$, which leads to \cref{eq:Pdecaz-Lk-ey0}.
			
			Invoking \cite[Lemma B.2]{Luo2022}, it is easy to conclude that 
			\[
			(1+\alpha_0)^{-1}(1+\alpha_1)^{-1}\cdots(1+\alpha_{k-1})^{-1}\leq \min\left\{\frac{4}{\gamma_0k^2},\,\left(1+\sqrt{\rho}\right)^{-k}\right\}.
			\]
			As \cref{eq:Pdecaz-Lk-ey0} implies $\mathcal L_k\leq \mathcal L_0\times 	(1+\alpha_0)^{-1}(1+\alpha_1)^{-1}\cdots(1+\alpha_{k-1})^{-1}$, we obtain \cref{eq:lk-rate} immediately. The final rate \cref{eq:final-rate} follows from the fact $\nm{y_k-y^*}_{A^\top(\epsilon)A(\epsilon)}^2=2(g(y_k)-g(y^*))\leq2\mathcal L_k$. This concludes the proof.
		\end{proof}
		
		\begin{rem}
			If $\gamma_0=\rho =\rho(\epsilon,\omega)$, then by \cref{thm:conv-ey0-ode-PNAG}, \cref{algo:kaacd} converges with a linear rate $O((1+\sqrt{\rho(\epsilon,\omega)})^{-k})$, which is an improvement of the rate $O((1-\rho(\epsilon,\omega))^{k})$ of \cref{algo:symkacd} (cf.\cref{thm:conv-kacd}). On the other hand, if we take $\rho =0$, then we get the  sublinear rate $O(1/k^2)$. This coincides with the optimal rates of accelerated gradient methods for convex and strongly convex objectives \cite{nesterov_lectures_2018}. Note that the work \cite{2015An} also present an accelerated (randomized) Kaczmarz method with the sublinear rate $O(1/k^2)$ and the linear rate $O((1+\sqrt{\rho_0}/(2m))^{-k}),\,\rho_0=\lambda^+_{\min}(A^\top(\epsilon)A(\epsilon))$, which is also an improvement of the classical Kaczmarz (cf.\cref{eq:rate-kacz}) but not robust for nearly singular system  since $\lambda^+_{\min}(A^\top(\epsilon)A(\epsilon))$ goes to 0 as $\epsilon\to0+$; see \cref{lem:up-bd-lambda-min+}.
		\end{rem}

\section{Numerical Experiments}
\label{sec:experiments}

In this section, we provide some numerical tests to validate the practical performance of our two main methods: the kernel-augmented coordinate descent \cref{eq:kacd} and the kernel-augmented accelerated coordinate descent (KaACD) (cf.\cref{algo:kaacd}).  The rest two methods: the kernel-augmented Kaczmarz \cref{eq:kak} and the symmetrized version of KaCD (cf.\cref{algo:symkacd}), are not reported here because (i) \cref{eq:kak} is impractical but its equivalent form \cref{eq:kacd} is practical and (ii) \cref{algo:symkacd} is served as a preconditioned coordinate descent subroutine for \cref{algo:kaacd}. 

For detailed comparison, we choose four baseline algorithms: 
\begin{itemize}
	\item coordinate descent (CD) cf.\cref{eq:cd-intro};
	\item accelerated coordinate descent (ACDM) \cite{nesterov_efficiency_2012};	
	\item randomized reshuffling Kaczmarz (RRK) \cite{han_simple_2025};	
	\item accelerated randomized Kaczmarz (ARK) \cite{2015An}.
\end{itemize}
Similarly with our KaCD and KaACD, the first two are coordinate type methods, but the other two are Kaczmarz type methods with random row action approach. Note also that ACDM itself uses randomly technique. In \cref{sec:6-1,sec:6-2}, we consider respectively a simple tridiagonal case and  the randomly generated data, and mainly focus on the comparison with CD. Then in \cref{sec:6-3}, we report the results of all methods on a sample matrix from a real-world data set. 

In all cases, the matrix $A(\epsilon)$ is nearly singular with respect to a small parameter $\epsilon$. The relaxation parameters are chosen as follows:  $\omega = 0.9 \times 2/\delta_{\max}(A(\epsilon))$ for KaCD and   $\omega= 0.9 \times 2/(1 + \delta_{\max}(A(\epsilon)))$ for KaCD and KaACD. For KaACD, we take the convexity parameter $\rho = 0.9 \times \rho_0(\omega) $, where $\rho_0(\omega)$ is given by \cref{eq:rhow}. For ACDM and ARK, the convexity parameters are the same $	\delta_{\min}(A(\epsilon))={}\lambda^+_{\min}(A^\top(\epsilon)D_\epsilon^{-1}A(\epsilon))$.
%For a given matrix \( A \in \mathbb{R}^{m \times n} \), we construct the consistent linear system \cref{eq:Ax=b} by setting \( b = A\cdot \text{rand}(n,1) \). 
%and $0<\rho_0(\omega)\leq\rho(\epsilon,\omega)\leq 1$.  
For all methods,  the initial guesses are zero vectors and the relative stopping  criteria are \( \|A(\epsilon)A^\top(\epsilon) y_k + b\|/\|b\| < 10^{-6} \) for coordinate type methods and \( \|A(\epsilon)x_k - b\|/\|b\|< 10^{-6} \) for Kaczmarz type methods.

%\subsection{Synthetic Data}
%\label{subsec:synthetic}
%
%We first examine the algorithms on synthetic examples designed to simulate nearly singular systems, where the condition number grows large as a parameter \( \epsilon \) approaches zero.

\subsection{Example 1}
\label{sec:6-1}

Consider the tridiagonal matrix in \cref{ex:eg3-8}.
%from \cite{lee_robust_2007}:
%\[
%A(\epsilon) = A(0) + \epsilon I, \quad \text{where } A(0) = 
%\begin{bmatrix}
%	1 & -1 & 0 \\
%	-1 & 2 & -1 \\
%	0 & -1 & 1
%\end{bmatrix}.
%\]
%The matrix \( A(0) \) has rank 2, and \( A(\epsilon) \) is nearly singular for small \( \epsilon \). 
%We test with values \( \epsilon = 1/5^k \) for \( k=1,2,3,4 \).
The numerical results are summarized in Table~\ref{tab:eg1} which show that KaCD significantly outperforms CD and the  accelerated variant KaACD further reduces number of iterations.
%, and both maintain robust performance as \( \epsilon \) decreases, while CD deteriorates sharply.

%\begin{figure}[H]
%	\centering
%	\begin{subfigure}{0.3\textwidth}
	%		\centering
	%		\includegraphics[width=\linewidth]{../fig/eg6.1.1sanduijiao/eg1_1CD_nokernel}
	%		\caption{\text{CD}}
	%		\label{../fig:eg1_a}
	%	\end{subfigure}
%	\hspace{0.5em}
%	\begin{subfigure}{0.3\textwidth}
	%		\centering
	%		\includegraphics[width=\linewidth]{../fig/eg6.1.1sanduijiao/eg1_1CD_kernel}
	%		\caption{KaCD}
	%		\label{../fig:eg1_b}
	%	\end{subfigure}
%	\hspace{0.5em}
%	\begin{subfigure}{0.3\textwidth}
	%		\centering
	%		\includegraphics[width=\linewidth]{../fig/eg6.1.1sanduijiao/eg1_1A_KACD}
	%		\caption{KaACD}
	%		\label{../fig:eg1_c}
	%	\end{subfigure}
%	\caption{Convergence behavior of CD, KaCD, and KaACD for solving the tridiagonal linear system with \( \epsilon = 1/5^k \) (Example~\ref{example:6.1}).}
%	\label{../fig:eg1}
%\end{figure}
\renewcommand{\arraystretch}{1.2}
\begin{table}[H]
	\centering
	%	\small
	\setlength{\tabcolsep}{4pt}
	\caption{Number of iteration of CD, KaCD, and KaACD for Example 1.}
	\label{tab:eg1}
	\begin{tabular}{@{}ccccc@{}}
		\toprule
		\(\epsilon\) & \(1/5\) & \(1/5^2\) & \(1/5^3\) & \(1/5^4\)  \\
		\midrule
		%		CD & 1003 & 20278 & 449809 & 9086890 \\
		CD & $1.0\times 10^3$ & $2.0\times 10^4$ & $4.5\times 10^5$ & $9.1\times 10^6$ \\
		%		\midrule
		KaCD  &   32  &  37   &  33  & 33       \\
		%		\midrule
		KaACD  &   20  &  21   &  20  &  20\\
		\bottomrule
	\end{tabular}
\end{table}

\subsection{Example 2}
\label{sec:6-2}
%We construct a random test matrix $A = (a_{ij}) \in \mathbb{R}^{m \times n}$ with $m < n$ and a target rank $r < m$. Let $\xi_i \stackrel{\text{i.i.d.}}{\sim} U(0,1)$ for $i=1,\dots,m$, and $U_{ij} \stackrel{\text{i.i.d.}}{\sim} U(0,1)$ be independent uniform random variables. We set $\epsilon = 1/3^k.$ Define the entries of $A$ as follows:
%1. First column: $a_{i1} = 1 + \xi_i$, $i=1,\dots,m$.
%2. Columns $j=2,\dots,m$:
%\[
%a_{ij} =
%\begin{cases}
%	1, & i = j \text{ and } i \leq r, \\
%	\epsilon, & i = j \text{ and } i > r, \\
%	\epsilon, & i \neq j \text{ and } U_{ij} > 0.5, \\
%	0, & \text{otherwise}.
%\end{cases}
%\]
%3. Columns $j=m+1,\dots,n$:
%\[
%a_{ij} =
%\begin{cases}
%	\epsilon, & U_{ij} > 0.8, \\
%	0, & \text{otherwise}.
%\end{cases}
%\]
%Then $\dim\dappker{A(\epsilon)} = m -r $. 
To satisfy \cref{assum:A0-A1,def:nsp}, we generate $A(\epsilon)$ randomly in three scenarios:
\begin{itemize}
	\item Case (a): \( m=50, n=80, r=45, \dim\dappker{A(\epsilon)} =5 \).
	\item Case (b): \( m=50, n=800, r=40,\dim\dappker{A(\epsilon)} =10 \).
	\item Case (c): \( m=30, n=5000, r=20,\dim\dappker{A(\epsilon)} =10 \).
\end{itemize}
All results are given in \cref{tab:example_6_1_4,tab:example_6_1_4_800,tab:example_6_1_4_5000}. Again, our kernel-augmented methods are highly robust and the accelerated variant KaACD performs much better than CD and KaCD. 
%\begin{figure}[H]
%	\centering
%	\begin{subfigure}{0.3\textwidth}
	%		\centering
	%	\includegraphics[width=\linewidth]{../fig/eg6.1.4m=50,n=80/eg6.1.4m=50,n=80,CD}
	%		\caption{CD}
	%		\label{../fig:eg6_1_4_a}
	%	\end{subfigure}
%	\hspace{0.5em}
%	\begin{subfigure}{0.3\textwidth}
	%		\centering
	%	\includegraphics[width=\linewidth]{../fig/eg6.1.4m=50,n=80/eg6.1.4m=50,n=80,KACD}
	%		\caption{KaCD}
	%		\label{../fig:eg6_1_4_b}
	%	\end{subfigure}
%	\hspace{0.5em}
%	\begin{subfigure}{0.3\textwidth}
	%		\centering
	%		\includegraphics[width=\linewidth]{../fig/eg6.1.4m=50,n=80/eg6.1.4m=50,n=80AKACD}
	%		\caption{KaACD}
	%		\label{../fig:eg6_1_4_c}
	%	\end{subfigure}
%	\caption{Convergence behavior of CD, KaCD, and  KaACD for solving the linear system with \( m=50, n=80, \epsilon=1/3^k \) (Case (a) of Example~\ref{example:6.3}).}
%	\label{../fig:example_6_1_4}
%\end{figure}
%\renewcommand{\arraystretch}{1.2}
\begin{table}[H]
	\centering
	%	\small
	\setlength{\tabcolsep}{6pt}
	\caption{Numerical results of CD, KaCD, and KaACD for Case (a) of Example 2.}
	\label{tab:example_6_1_4}
	\begin{tabular}{@{}cccccc@{}}
		\toprule
		\(\epsilon\) & \(1/3\) & \(1/3^2\) & \(1/3^3\) & \(1/3^4\) & \(1/3^5\) \\
		\midrule
		CD & 2983 & 7846 & 8201 & 27752 & 89596 \\
		%		\midrule
		KaCD & 1940 & 2522 & 2315 & 1614 & 1556 \\
		%		\midrule
		KaACD & 300 & 276 & 244 & 248 & 238 \\
		\bottomrule
	\end{tabular}
\end{table}

%\begin{figure}[H]
%	\centering
%	\begin{subfigure}{0.3\textwidth}
	%		\centering
	%		\includegraphics[width=\linewidth]{../fig/eg6.1.4m=50,n=800/eg6.1.4m=50,n=800,CD}
	%		\caption{CD}
	%		\label{../fig:eg6_1_4_800_a}
	%	\end{subfigure}
%	\hspace{0.5em}
%	\begin{subfigure}{0.3\textwidth}
	%		\centering
	%		\includegraphics[width=\linewidth]{../fig/eg6.1.4m=50,n=800/eg6.1.4m=50,n=800,KACD}
	%		\caption{KaCD}
	%		\label{../fig:eg6_1_4_800_b}
	%	\end{subfigure}
%	\hspace{0.5em}
%	\begin{subfigure}{0.3\textwidth}
	%		\centering
	%		\includegraphics[width=\linewidth]{../fig/eg6.1.4m=50,n=800/eg6.1.4m=50,n=800,A_KACD}
	%		\caption{KaACD}
	%		\label{../fig:eg6_1_4_800_c}
	%	\end{subfigure}
%	\caption{Convergence behavior of CD, KaCD, and KaACD for solving the linear system with \( m=50, n=800, \epsilon=1/3^k \) (Case (b) of Example~\ref{example:6.3}).}
%	\label{../fig:example_6_1_4_800}
%\end{figure}

\begin{table}[H]
	\centering
	%	\small
	\setlength{\tabcolsep}{6pt}
	\caption{Numerical results of CD, KaCD, and KaACD for Case (b) of Example 2.}
	\label{tab:example_6_1_4_800}
	\begin{tabular}{@{}cccccc@{}}
		\toprule
		\(\epsilon\) & \(1/3\) & \(1/3^2\) & \(1/3^3\) & \(1/3^4\) & \(1/3^5\) \\
		\midrule
		CD & 182 & 541 & 2820 & 16141 & 60767 \\
		%		\midrule
		KaCD & 172 & 477 & 1571 & 1594 & 1588 \\
		%		\midrule
		KaACD & 79 & 140 & 239 & 242 & 239 \\
		\bottomrule
	\end{tabular}
\end{table}

%\begin{figure}[H]
%	\centering
%	\begin{subfigure}{0.3\textwidth}
	%		\centering
	%		\includegraphics[width=\linewidth]{../fig/eg6.1.4m=30,n=5000/eg6.1.4m=30,n=5000,CD}
	%		\caption{CD}
	%		\label{../fig:eg6_1_4_5000_a}
	%	\end{subfigure}
%	\hspace{0.5em}
%	\begin{subfigure}{0.3\textwidth}
	%		\centering
	%		\includegraphics[width=\linewidth]{../fig/eg6.1.4m=30,n=5000/eg6.1.4m=30,n=5000,KACD}
	%		\caption{KACD}
	%		\label{../fig:eg6_1_4_5000_b}
	%	\end{subfigure}
%	\hspace{0.5em}
%	\begin{subfigure}{0.3\textwidth}
	%		\centering
	%		\includegraphics[width=\linewidth]{../fig/eg6.1.4m=30,n=5000/eg6.1.4m=30,n=5000,AKACD}
	%		\caption{KaACD}
	%		\label{../fig:eg6_1_4_5000_c}
	%	\end{subfigure}
%	\caption{Convergence behavior of CD, KaCD, and KaACD for solving the linear system with \( m=30, n=5000, \epsilon=1/3 \) (Case (c) of Example~\ref{example:6.3}).}
%	\label{../fig:example_6_1_4_5000}
%\end{figure}

\begin{table}[H]
	\centering
	%	\small
	\setlength{\tabcolsep}{6pt}
	\caption{Numerical results of CD, KaCD, and KaACD for Case (c) of Example 2.}
	\label{tab:example_6_1_4_5000}
	\begin{tabular}{@{}cccccc@{}}
		\toprule
		\(\epsilon\) &  \(1/3^2\) & \(1/3^3\) & \(1/3^4\) & \(1/3^5\)&  \(1/3^6\)\\
		\midrule
		CD &   85 & 360 & 1843 & 12983&56771 \\
		%		\midrule
		KaCD &   97 & 310 & 979 & 907 &903\\
		%		\midrule
		KaACD &   49 & 109 & 179 & 176 &176\\
		\bottomrule
	\end{tabular}
	%	\begin{tabular}{@{}ccccccc@{}}
		%		\toprule
		%		\(\epsilon\) & \(1/3\) & \(1/3^2\) & \(1/3^3\) & \(1/3^4\) & \(1/3^5\)&  \(1/3^6\)\\
		%%		\midrule
		%		CD & 54 & 85 & 360 & 1843 & 12983&56771 \\
		%%		\midrule
		%		KaCD & 62 & 97 & 310 & 979 & 907 &903\\
		%%		\midrule
		%		KaACD & 33 & 48 & 95 & 178 & 182 &172\\
		%		\bottomrule
		%	\end{tabular}
\end{table}

\subsection{Example 3}
\label{sec:6-3}
This example chooses a sample matrix from the {\it suitesparse matrix collection} \cite{kolodziej2019suitesparse}. We start from a $453 \times 453$ matrix and extract a full-row-rank submatrix $A \in \mathbb{R}^{50 \times 453}$. To generate test problems with adjustable rank deficiency, we define a parameterized matrix $A(\epsilon)$ via the singular value decomposition of $A = U \Sigma V^\top$:
\[
A(\epsilon) = U \Sigma(\epsilon) V^\top, \quad \text{where } \Sigma(\epsilon)_{ii} = \begin{cases}
	\epsilon \cdot \Sigma_{ii}, & \text{for the 5 smallest singular values},\\
	\Sigma_{ii}, & \text{otherwise}.
\end{cases}
\]
We set $\epsilon=1/2^k, k=1,2,...,6.$ As $\epsilon$ goes to 0, the rank of $A(\epsilon)$ drops abruptly from 50 to 45. 

%For coordinate type mehtods, the true solution is set as \( y^* = \text{randn}(m,1) \), and the right-hand side is computed as \( b = A(\epsilon)A(\epsilon)^\top y^* \) where $y^*$ is the exact solution. For Kaczmarz type methods, the true solution is set as \( x^* = \text{randn}(n,1) \), and the right-hand side is computed as \( b = A(\epsilon) x^* \) where $x^*$ is the exact solution.  We compare five algorithms for solving this linear system: the randomized reshuffling Kaczmarz (RRK) \cite{han_simple_2025}, the accelerated randomized Kaczmarz (ARK) \cite{2015An},  accelerated coordinate descent\cite{nesterov_efficiency_2012} (ACDM), KaCD, and KaACD. Parameters for ACMD: $v_0 = y_0 = 0\in \mathbb{R}^m , a_0=1/m, b_0=2,\gamma_0=1/m $. Following \cite{nesterov_efficiency_2012}, the strong convexity parameter $\sigma$ required for the ACDM can be obtained by solving the generalized eigenvalue problem $(A(\epsilon) A(\epsilon)^\top) v = \lambda D_{\epsilon} v$. We then set $\sigma$ to the smallest generalized eigenvalue: $\sigma = \lambda_{\min}((A(\epsilon) A(\epsilon)^\top, D_{\epsilon})$.
%  Parameters for ARK: $v_0 = x_0 = 0 \in \mathbb{R}^n, \gamma_{0}=0,\lambda = \lambda_{\min}^{+}((A(\epsilon) A(\epsilon)^\top).$
%Han and Xie established a linear convergence analysis for the randonmized reshuffling Kaczmarz (RRK) method  \cite{han_simple_2025}, proving that RRK converges linearly to the minimum-norm solution with a tight dimension-independent bound.
%\kk{这部分文献要在intro里面加}

Table~\ref{tab:realworld} presents the numerical outputs. Since RRK, ARK and ACDM are randomized algorithms, their iteration counts have been divided by 50, i.e., the number of rows of $A(\epsilon)$. 
The results are striking: standard methods (RRK, ARK, CD, ACDM) are highly sensitive to the ill-conditioning property caused by small singular values. In contrast, both KaCD and KaACD show high robustness and the number of iterations leaves almost invariant. Particularly, KaACD is more efficient than KaCD and the iteration steps is about half of KaCD.
\begin{table}[H]
	\centering
	%	\small
	\setlength{\tabcolsep}{6pt}
	\caption{Numerical results of Example 3.}
	%	\caption{Iteration counts of six algorithms for solving the real-world linear system with varying \( \epsilon \) values under stopping criterion: \( \|AA^\top y_k + b\|/\|b\| < 10^{-7} \) for coordinate type methods and \( \|Ax_k - b\|/\|b\|< 10^{-7} \) for Kaczmarz type methods.}
	\label{tab:realworld}
	\begin{tabular}{@{}ccccccc@{}}
		\toprule
		\(\epsilon\) & \(1/2\) & \(1/2^2\) & \(1/2^3\) & \(1/2^4\) & \(1/2^5\) & \(1/2^6\) \\
		\midrule
		%		CD & 577 & 1343 & 2266 & 8583 & 14198 & 17319 \\
		%		\bf{KaCD} & 322 & 297 & 290 & 302 & 272 & 312 \\
		%		KaACD & 97 & 125 & 106 & 105 & 108 & 94 \\
		CD & 577 & 1343 & 2266 & 8583 & 14198 & 17319 \\      						
		\bf{KaCD} & \bf{409} & \bf{389} & \bf{398} & \bf{411} & \bf{370} & \bf{368} \\
		%\hline
		\bf{KaACD} & \bf{150} & \bf{164} & \bf{171} & \bf{177} & \bf{168} & \bf{173} \\	
		%		RRK\cite{han_simple_2025} & 18 & 67 & 217 & 821 & 2578 & 9645\\
		%			ARK \cite{2015An} & 562 & 1109 & 2395 & 4374 & 8451 & 12149  \\
		ARK \cite{2015An} & 609 & 1098 & 2262 & 4509 & 9309 & 13413\\        
		ACDM \cite{nesterov_efficiency_2012} & 946 & 2418 & 5804 & 9618 & 13608 & 30523  \\
		RRK \cite{han_simple_2025} & 1207 & 3714 & 13642 & 52789 & 210342 & 884266\\		
		\bottomrule
	\end{tabular}
\end{table}

\section{Conclusion}
\label{conclu}
In this work, we developed a robust kernel-augmented Kaczmarz (KaK) for nearly singular linear systems. The rate of convergence is uniformly stable even if the smallest nonzero singular value goes to zero. Also, we present two kernel-augmented coordinate descent methods as implementable and accelerated variants of KaK. Numerical results
are provided to validate the performance of the proposed methods compared with existing Kaczmarz type methods and coordinate type methods. An interesting future topic is extending our current framework to the linearly constrained case; see \cite{Luo2021,zeng_stochastic_2026}.

\appendix
\section{The Auxiliary Space Lemma}
The following vital lemma has already been proved by \cite[Lemma 2.4]{xu_method_2002}. We will provide an alternative proof from an optimization perspective.
\begin{lem}[\bf Auxiliary space lemma]
	\label{lem:inv-batB}
	Let $\boldsymbol{V}$ and $V$ be two Hilbert spaces with the inner products $(\cdot,\cdot)_{\boldsymbol{V}}$ and $(\cdot,\cdot)_{V}$. Assume $\mathcal{R}: \boldsymbol{V} \to V$ is surjective and denote by $\mathcal{R}^\star:V\to \bm V$ the adjoint operator of $\mathcal R$:
	\[
	\inner{\mathcal R^\star u,\bm v}_{\bm V}=\inner{u,\mathcal R\bm v}\quad\forall\,\bm v\in\bm V,u\in V.
	\]
	If $\boldsymbol{B}: \boldsymbol{V} \to \boldsymbol{V}$ is SPD then so is $B := \mathcal{R} \boldsymbol{B} \mathcal{R}^\star$ and
	\begin{equation}\label{eq:aux-lem}
		\left(B^{-1}v,v\right)_{V} = \inf_{\mathcal{R}\boldsymbol{v}=v} \left(\boldsymbol{B}^{-1}\boldsymbol{v},\boldsymbol{v}\right)_{\boldsymbol{V}} \quad \forall\,v\in V,
	\end{equation}
	where the infimum is attained uniquely at $\widehat{\boldsymbol{v}} = \boldsymbol{B}\mathcal{R}^\star B^{-1}v$.
\end{lem}

\begin{proof}
	It is easy to see $B$ is symmetric and invertible and thus it is SPD. Treat the right hand side of \cref{eq:aux-lem} as an equality constrained convex quadratic optimization, the Lagrangian to which is
	\[
	\mathcal{E}(\boldsymbol{v},\lambda ):=\left(\boldsymbol{B}^{-1}\boldsymbol{v},\boldsymbol{v}\right)_{\boldsymbol{V}}-\left(\lambda ,\mathcal{R}\boldsymbol{v}-v\right)_{V}\quad\forall\,(\boldsymbol{v},\lambda )\in\boldsymbol{V}\times V.
	\]
	Look at the Euler-Lagrangian equations $	2\boldsymbol{B}^{-1}\widehat{\boldsymbol{v}} - \mathcal{R}^\star\widehat{\lambda } = 0$ and $\mathcal{R}\widehat{\boldsymbol{v}} = v$. With some elementary manipulations, we get the unique optimal solution $\widehat{\boldsymbol{v}}=\boldsymbol{B}\mathcal{R}^\star B^{-1}v$ and the Lagrange multiplier $\widehat{\lambda }=2B^{-1}v$, which implies immediately that
	\[
	\inf_{\mathcal{R}\boldsymbol{v}=v}\left(\boldsymbol{B}^{-1}\boldsymbol{v},\boldsymbol{v}\right)_{\boldsymbol{V}}=\left(\boldsymbol{B}^{-1}\widehat{\boldsymbol{v}},\widehat{\boldsymbol{v}}\right)_{\boldsymbol{V}}=\left(B^{-1}v,v\right)_{V}.
	\]
	This completes the proof of this lemma.
\end{proof}

\bibliographystyle{plain} 
%		\bibliographystyle{abbrv}
%\bibliography{../refs} 

\end{document}